%% file: ainf.tex
\documentclass[12pt,reqno]{amsart}

\usepackage[pdfauthor   = {Sebastian\ Posur},
            pdftitle    = {A\ constructive\ approach\ to\ Fourier-Mukai\ transforms},
            pdfsubject  = {},
            pdfkeywords = {Fourier Mukai transforms, dg categories, A infinity functors},
            bookmarks=true,
            bookmarksopen=true,
            pagebackref=true,
            hyperindex=true,
            colorlinks=true,
            linkcolor=blue,
            citecolor=blue,
            filecolor=blue,
            urlcolor=blue,
            ]{hyperref}

\include{header}

\author{Sebastian Posur}
\thanks{The author is supported by Deutsche Forschungsgemeinschaft (DFG) grant SFB-TRR 195: \emph{Symbolic Tools in Mathematics and their Application}}
\address{Department of mathematics, University of Siegen, 57072 Siegen, Germany}
\email{\href{mailto:Sebastian Posur <sebastian.posur@uni-siegen.de>}{sebastian.posur@uni-siegen.de}}

\begin{document}

\title[A constructive approach to Fourier-Mukai transforms]{A constructive approach to Fourier-Mukai transforms for projective spaces via \Ainf-functors between pretriangulated dg categories}

\begin{abstract}
We discuss the following problem: 
how can an arbitrary
Fourier-Mukai transform 
$\phi: \mathrm{D}^b( \mathbb{P}^a ) \rightarrow \mathrm{D}^b( \mathbb{P}^b )$
between the bounded derived categories 
of two projective spaces of dimensions $a$ and $b$
be expressed in explicit terms as an exact functor between
the homotopy categories 
$\mathrm{K}^b( \mathbb{B}^a ) \rightarrow \mathrm{K}^b( \mathbb{B}^b )$
generated by the full strong exceptional sequences of the line bundles
$\mathbb{B}^a = \{\mathcal{O}(-a), \dots, \mathcal{O}\}$ and $\mathbb{B}^b = \{\mathcal{O}(-b), \dots, \mathcal{O}\}$?
We show that this problem can be 
reduced to the following task which is independent of any prescribed Fourier-Mukai kernel:
finding an \Ainf-functor $P$ in explicit terms whose induced functor on homotopy categories
yields the embedding of
$\{ \mathcal{O}(i) \boxtimes \mathcal{O}(j) \mid i = -2a, \dots, 0, ~~j = -2b, \dots 0 \}$
into $\Db( \Pro^{a} \times \Pro^{b})$.

As our main technical tool we provide an explicit formula for the lift of an \Ainf-functor
$F: \AC \rightarrow \BC$ between a dg category $\AC$ and a pretriangulated dg category $\BC$
to the pretriangulated hull of $\AC$
given by the universal property of pretriangulated hulls.
As a further application of this tool, we provide a simple example of
two non-isomorphic exact functors between triangulated categories that coincide on
the full subcategory generated by a full strong exceptional sequence.
\end{abstract}

\maketitle
\setcounter{tocdepth}{4}
\tableofcontents

\input{ainf_content.tex}

\input{ainf.bbl}

\end{document}

%% file: header.tex
\usepackage[utf8]{inputenc} 
\usepackage[T1]{fontenc}
\usepackage{a4wide}
\usepackage{lmodern}
\usepackage[english]{babel}
\usepackage{mathrsfs}
\usepackage{etex}
\usepackage{mathtools}
\usepackage{latexsym}
\usepackage{amssymb}
\usepackage{amsthm}
\usepackage{amsmath}
\usepackage{caption}

\usepackage{physics}
\usepackage{tensor}
\usepackage{colortbl}

\usepackage[all]{xy}
\usepackage{verbatim}
\usepackage{listings}
\usepackage{fancyvrb}

\usepackage{graphicx}

\usepackage[dvipsnames]{xcolor}
\usepackage{accents} 
\usepackage{enumerate}
\usepackage{wrapfig}
\usepackage{tikz}
\usepackage{tikz-cd}
\usetikzlibrary{automata,shapes,arrows,matrix,backgrounds,positioning,plotmarks,calc,patterns,matrix,decorations.pathreplacing,decorations.pathmorphing,decorations.text,decorations.markings}
\usepackage[colorinlistoftodos,shadow]{todonotes}

\usepackage{multirow}
\usepackage{mdwlist}

\usepackage{stmaryrd}
\usepackage{mathdots} 

\usepackage{fancyvrb}

\usepackage[toc,page]{appendix}
\usepackage{float}

\usepackage{extarrows}
\usepackage{pdflscape}
\usepackage{rotating}

\newtheoremstyle{mytheoremstyle} 
    {5pt}                    
    {5pt}                    
    {\itshape}                   
    {\parindent}                           
    {\bf}                   
    {.}                          
    {.5em}                       
    {}  

\theoremstyle{mytheoremstyle}

\newtheorem{theorem}{Theorem}[section]

\newtheorem{lemma}[theorem]{Lemma}

\newtheorem{corollary}[theorem]{Corollary}

\newtheoremstyle{mytdefintionstyle} 
    {5pt}                    
    {5pt}                    
    {\rm}                   
    {\parindent}                           
    {\bf}                   
    {.}                          
    {.5em}                       
    {}  

\theoremstyle{remark}
\newtheorem{remark}[theorem]{Remark}

\theoremstyle{mytdefintionstyle}
\newtheorem{definition}[theorem]{Definition}

\newtheorem{example}[theorem]{Example}

\newtheorem{construction}[theorem]{Construction}

\newtheorem{notation}[theorem]{Notation}
\newtheorem*{notationnonumber}{Notation}
\newtheorem*{acknowledgments}{Acknowledgments}

\newtheoremstyle{exmp_contd} 
{\topsep} {\topsep}%
{\upshape}
{}
{\bfseries}
{}
{ }
{\thmname{#1}\,\thmnumber{ #2}\thmnote{#3}\enspace(continued)}

\theoremstyle{exmp_contd}

\usepackage{xspace}
\input{pre.tex}

\tikzset{round left paren/.style={ncbar=0.5cm,out=120,in=-120}}
\tikzset{round right paren/.style={ncbar=0.5cm,out=60,in=-60}}
\newcolumntype{C}[1]{>{\centering\arraybackslash$}p{#1}<{$}}
\newlength{\mycolwd}
\usepackage{array, xcolor}
\definecolor{lightgray}{gray}{0.8}
\newcolumntype{L}{>{\raggedleft}p{0.28\textwidth}}
\newcolumntype{R}{p{0.8\textwidth}}

\definecolor{ctcolor}{gray}{0.95}

\definecolor{ctucolor}{gray}{0.85}

\makeatletter
\newcommand{\thickhline}{%
    \noalign {\ifnum 0=`}\fi \hrule height 1pt
    \futurelet \reserved@a \@xhline
}
\newcolumntype{"}{@{\hskip\tabcolsep\vrule width 1pt\hskip\tabcolsep}}
\makeatother

%% file: pre.tex
\definecolor{ExQ}{HTML}{0000FF}
\definecolor{Dec}{HTML}{E07B00}

\newcommand{\CapPkg}{\textsc{Cap}\xspace}

\newcommand{\Ainf}{\texorpdfstring{$A_{\infty}$}{A-infinity}}
\newcommand{\AC}{\mathbf{A}}
\newcommand{\BC}{\mathbf{B}}

\newcommand{\GC}{\mathbf{G}}
\newcommand{\QC}{\mathbf{Q}}
\newcommand{\TC}{\mathbf{T}}

\newcommand{\N}{\mathbb{N}}
\newcommand{\Z}{\mathbb{Z}}

\newcommand{\Pro}{\mathbb{P}}

\newcommand{\CH}{\mathrm{H}}

\newcommand{\OS}{\mathcal{O}}

\newcommand{\Coh}{\mathfrak{Coh}}

\newcommand{\Ch}{\mathrm{Ch}}

\newcommand{\Db}{\mathrm{D}^{\mathrm{b}}}
\newcommand{\Kb}{\mathrm{K}^{\mathrm{b}}}
\newcommand{\KCh}{\mathrm{K}}
\newcommand{\Chb}{\mathrm{Ch}^{\mathrm{b}}}
\newcommand{\Chbdg}{\mathrm{Ch}_{\mathrm{dg}}^{\mathrm{b}}}
\newcommand{\Dbdg}{\mathrm{D}_{\mathrm{dg}}^{\mathrm{b}}}

\newcommand{\LDer}{\mathbf{L}}
\newcommand{\RDer}{\mathbf{R}}
\newcommand{\otimesL}{\stackrel{\LDer}{\otimes}}
\newcommand{\IHomR}{ \RDer\mathrm{\underline{Hom}}}

\newcommand{\opC}{\mathrm{op}}

\newcommand{\Obj}{\mathrm{Obj}}

\DeclareMathOperator{\dgCat}{\mathbf{dgCat}}
\DeclareMathOperator{\Hqe}{\mathbf{Hqe}}
\DeclareMathOperator{\HCat}{\mathbf{HCat}}

\newcommand{\dg}{\mathrm{dg}}

\DeclareMathOperator{\Hzero}{\mathrm{H}^0}

\newcommand{\id}{\mathrm{id}}

\DeclareMathOperator{\End}{\mathrm{End}}
\DeclareMathOperator{\Hom}{\mathrm{Hom}}
\DeclareMathOperator{\IHom}{\mathrm{\underline{Hom}}}

\DeclareMathOperator{\Ext}{\mathrm{Ext}}

\newcommand{\pretr}{\mathrm{pretr}}
\newcommand{\Twist}{\mathrm{Twist}}

\newcommand{\Modl}{\text{-}\mathrm{Mod}}
\newcommand{\Modr}{\mathrm{Mod}\text{-}}

\newcommand{\Diag}{\mathrm{Diag}}

\newcommand{\degdown}[1]{ \underline{#1} }
\newcommand{\degup}[1]{ \overline{#1} }

\newcommand{\transleft}[1]{ \tensor[_{[#1]}]{{(-)}}{} }
\newcommand{\transright}[1]{\tensor[]{{(-)}}{_{[#1]}} }
\newcommand{\trans}[3]{ \tensor[_{[#1]}]{{#2}}{_{[#3]}} }

\newcommand{\twistcleft}[1]{ \tensor[_{  \{#1\}  }]{{(-)}}{} }
\newcommand{\twistcright}[1]{\tensor[]{{(-)}}{_{  \{#1\} }} }
\newcommand{\twistc}[3]{ \tensor[_{ \{#1\} }]{{#2}}{_{ \{#3\} }} }

\newcommand{\twistcobj}[2]{ {#1}\{{#2}\} }

\newcommand{\resclass}[1]{ [{#1}] }

\newcommand{\BLineBundles}[1]{ \mathbb{B}^{#1} }
\newcommand{\BLineBundlesZero}[1]{  {\mathbb{B}_{0}^{#1}} }

\newcommand{\LineBundles}{ \mathbb{B} }
\newcommand{\LineBundlesab}{ \BLineBundles{a,b} }
\newcommand{\LineBundlesabdouble}{ \widehat{\mathbb{B}}^{a,b} }

%% file: ainf_content.tex
\section{Introduction}\label{section:introduction}

Fourier-Mukai transforms were introduced by Mukai in \cite{Mukai81}
for the study of abelian varieties.
They form a very useful class of exact functors 
$\phi: \Db( X ) \rightarrow \Db( Y )$
between the bounded derived categories of smooth projective varieties $X,Y$ over a field $k$.
Lots of exact functors which occur naturally in algebraic geometry are Fourier-Mukai transforms,
in particular all fully faithful exact functors by Orlov’s representability theorem \cite{Orl97},
but also counter-examples are part of the discussion \cite{CSdg17, BNR19}.
In any case, Fourier-Mukai transforms provide a widely used tool in modern algebraic geometry \cite{Huy_FM}.

It is a challenge in computer algebra to render Fourier-Mukai transforms constructive, which means
to find data structures for objects and morphisms in $\Db( X )$, $\Db( Y )$,
and an algorithm for computing $\phi$ on a specific input.
In the case of products of projective spaces, the direct image functor of a projection to a factor
can be computed using the machinery of Tate resolutions \cite{EESProducts15, EESProductsMc2},
where the data structure for the objects is based on finitely presented modules over the homogeneous coordinate ring.

For a projective space $\Pro^n$ of dimension $n$ over $k$, let $\BLineBundles{n}$ denote the full subcategory of $\Db( \Pro^n )$
spanned by the line bundles $\{\OS(-i)\}_{i = 0,\dots,n}$.
Then the homotopy category $\Kb( \BLineBundles{n} )$ of complexes whose objects consist of direct sums of line bundles in $\BLineBundles{n}$
is equivalent to $\Db( \Pro^n )$ \cite{Bei78}. In fact, the objects in $\BLineBundles{n}$ yield a full strong exceptional collection,
and thus generate $\Db( \Pro^n )$ as a triangulated category.
In this paper, we study the question of how to express a given Fourier-Mukai transform $\phi: \Db( \Pro^a ) \rightarrow \Db( \Pro^b )$
for $a,b \in \N_0$
in explicit terms as an exact functor $\Phi: \Kb( \BLineBundles{a} ) \rightarrow \Kb( \BLineBundles{b} )$,
motivated by the fact that the category $\Kb( \BLineBundles{a} )$ allows for explicit computations.

It is a tempting idea that it should suffice to describe $\Phi$ only on some kind of generators (e.g., the subcategory $\BLineBundles{a}$)
and then ``naturally'' extend $\Phi$ to the whole category $\Kb( \BLineBundles{a} )$.
In fact, we prove in Section \ref{section:products} that two of three ingredients of a Fourier-Mukai transform defining $\phi$
(namely the pullback and the direct image)
can indeed be modeled by following this tempting idea in a naive manner, see Subsections \ref{subsection:pullback} and \ref{subsection:direct_image}.
However, for the third ingredient, the tensor product, this strategy fails, see Subsection \ref{subsection:difficulty}.

The failure of naively defining $\Phi$ on generators originates from the well-known
deficiency of triangulated categories: cones are not functorial.
To overcome this failure, dg enhancements of triangulated categories were invented \cite{BonKap90}.
Roughly, a dg category $\AC$ is a category whose homomorphism sets are given by complexes,
and passing to the $0$-th cohomology yields an ordinary category $\Hzero( \AC )$, called the homotopy category of $\AC$.
A dg category is called pretriangulated if $\Hzero( \AC )$ is triangulated in a natural sense.
In our case, $\Kb( \BLineBundles{a} )$ has a dg enhancement given by $\Chbdg( \BLineBundles{a} )$, the pretriangulated dg category of complexes
whose objects are direct sums of line bundles in $\BLineBundles{a}$.
The appropriate notion of a homotopy adapted functor between dg categories is provided by the machinery of \Ainf-categories.
We summarize the theory of dg categories and \Ainf-functors for our purposes in Section \ref{section:dg_ainf_theory}.

In the language of pretriangulated dg categories and \Ainf-functors,
specifying functors simply on generators works as expected.
If we have an \Ainf-functor $F: \AC \rightarrow \BC$ from a dg category into a pretriangulated dg category,
it can be lifted to an \Ainf-functor
\[
    F^{\sharp}: \pretr(\AC) \rightarrow \BC
\]
starting from the so-called pretriangulated hull $\pretr(\AC)$ of $\AC$
in an essentially unique manner (up to quasi-equivalence).
In Section \ref{section:constructive}, we provide
explicit constructions of both the dg category $\pretr( \AC )$ (Subsection \ref{subsubsection:pretr})
and an explicit formula for its universal property (Theorem \ref{theorem:lift_theorem}).
The degree of explicitness in our exposition is chosen on a level such that both a computer implementation
and an understanding on the human level is possible, in order to facilitate
computations with such functors in the future.
In particular, we build up $\pretr( \AC )$ by means of category constructors applied to $\AC$:
the completion by direct sums $\AC \mapsto \AC^{\oplus}$ (Subsection \ref{subsubsection:direct_sums}),
the completion by translations $\AC \mapsto \AC^{[\bullet]}$ (Subsection \ref{subsubsection:translations}),
and the completion by twisted complexes $\AC \mapsto \Twist( \AC )$ (Subsection \ref{subsubsection:twisted_complexes}).
Then $\pretr( \AC )$ arises as the dg subcategory of $\Twist( (\AC^{[\bullet]})^{\oplus} )$ formed by so-called one-sided twisted complexes.
We also follow this tower of category constructors in order to lift $F$ to the desired $F^{\sharp}$,
i.e., first we lift $F$ to the completion by direct sums (Subsection \ref{subsubsection:lift_direct_sums}),
second to the completion by translations (Subsection \ref{subsubsection:lift_translations}),
and last to the pretriangulated hull (Subsection \ref{subsubsection:lift_pretr}).
The descriptions of these lifts form the technical heart of this paper.

As a first application of our formulas,
we give a short example of triangulated categories $\TC_1$, $\TC_2$
and two exact functors $F,G: \TC_1 \rightarrow \TC_2$
which coincide on the full subcategory spanned by a full strong exceptional sequence,
but for which we have $F \not\simeq G$ (Corollary \ref{corollary:warning}).
The provided counter-example makes the idea that concrete choices of homotopies
may actually matter within the application of $F$ or $G$ to an object in $\TC_1$
very tangible, and thus fits nicely into the general philosophy that ``proof data'' may be actually relevant
within mathematical constructions \cite{hottbook}.

In our second application, we come back to our original problem of modeling Fourier-Mukai transformations for projective spaces.
In Theorem \ref{theorem:fm_for_pn}
we state that modeling a Fourier-Mukai transform $\phi: \Db( \Pro^a ) \rightarrow \Db( \Pro^b )$
in terms of a functor $\Kb( \BLineBundles{a} ) \rightarrow \Kb( \BLineBundles{b} )$ boils down to a single ingredient, namely
an \Ainf-functor
\[
    P: \LineBundlesabdouble \rightarrow \pretr( \LineBundlesab )
\]
from the full subcategory $\LineBundlesabdouble \subseteq \Db( \Pro^a \times \Pro^b )$ generated by the line bundles
$\OS(i) \boxtimes \OS(j)$ for $i = -2a, \dots, 0$, $j = -2b, \dots 0$,
to the pretriangulated hull of the 
full subcategory $\LineBundlesab \subseteq \Db( \Pro^a \times \Pro^b )$ generated by the line bundles
$\OS(i) \boxtimes \OS(j)$ for $i = -a, \dots, 0$, $j = -b, \dots 0$,
such that $\Hzero( P )$ corresponds to the natural inclusion $\LineBundlesabdouble \hookrightarrow \Db( \Pro^a \times \Pro^b )$.

\begin{notationnonumber}
    We fix a field $k$.
    Whenever we deal with a $\Z$-graded $k$-vector space $V$ we will simplify terminology as follows:
    \begin{enumerate}
        \item We call such a $V$ simply a \emph{graded space}, and likewise a $\Z$-graded $k$-linear map between graded spaces a \emph{graded map}.
              We refer to the degree $d \in \Z$ component of $V$ by $V^d$.
        \item By an element $v \in V$ we always mean a \emph{homogeneous} element, i.e.,
              an element $v \in V^d$ for some $d \in \Z$. We write $\abs{v} := d$ for its degree.
        \item Whenever we define or deal with a graded map, we will use or write down application rules $v \mapsto \alpha(v)$ only for the homogeneous elements.
    \end{enumerate}
    Given two graded spaces $V$, $W$,
    their \emph{tensor product} is given by the graded space
    \[ V \otimes W := \bigoplus_{i \in \Z}\left( \bigoplus_{d}( V^d \otimes W^{i-d}) \right).\]
    We use \emph{Fukaya's sign convention}:
    suppose given two graded maps $\alpha: V \rightarrow V'$, $\beta: W \rightarrow W'$,
    their tensor product is defined by
    \[
    (\alpha \otimes \beta)( v \otimes w ) := (-1)^{\abs{v} \abs{\beta}}\alpha( v ) \otimes \beta( w )
    \]
    for all $v \in V$, $w \in W$.
    Since the theory presented in this paper heavily relies on a correct management of signs,
    we need to introduce explicit notations for shifts:
    the shift $SV$ of a graded space $V$ is given by $(SV)^d := V^{d+1}$.
    An element $v \in V^d$ can be considered as an element in $(SV)^{d-1}$,
    and we will make this explicit via the graded map
    \[
        \degdown{(-)}: V \rightarrow SV: v \mapsto v
    \]
    of degree $-1$.
          Thus, whenever we have $v \in V$ and want to consider it inside
          $SV$, we will write $\degdown{v} \in SV$ (we lower the degree: $\abs{\degdown{v}} = \abs{v} - 1$).
          Applying the inverse of this map to an element $w \in SV$ will be denoted by $\degup{w}$
          (we increase the degree: $\abs{\degup{w}} = \abs{w} + 1$).
          
    We will write $k$-category for a category enriched over $k$-vector spaces.
    All functors between $k$-categories are assumed to be $k$-linear.
    If $\AC$ is a $k$-category, a dg category, or an \Ainf-category, and $A,B \in \AC$ objects,
    we use the notations $\Hom_{\AC}(A,B)$, $(A,B)_{\AC}$ or simply $(A,B)$ in order to refer to the homomorphisms between $A$ and $B$,
    endowed with their appropriate extra structure (a $k$-linear space or a cochain complex).
    A $k$-category $\AC$ can always be regarded as a dg category (Construction \ref{construction:k_linear_to_dg}),
    which in turn can be regarded as an \Ainf-category (Construction \ref{construction:dg_to_Ainf}).
    We will tacitly make use of these conversions throughout the paper, e.g., when we speak about \Ainf-functors between dg categories.
    
    We use cohomological conventions, in particular, by a complex we mean a cochain complex.
    The homotopy class of a cochain map $\alpha$ is denoted by $\resclass{\alpha}$.
    If $\AC$ is a $k$-category, we denote by $\Chb( \AC )$
    the category of bounded complexes whose objects consist of formal direct sums of objects in $\AC$,
    and whose differentials are given by matrices whose entries take values in the morphisms in $\AC$.
    If $\AC$ already has direct sums, $\Chb( \AC )$ is the usual category of bounded complexes.
    We denote the corresponding homotopy category by $\Kb( \AC )$.
    
    We denote the set of natural numbers including zero by $\N_0$.
\end{notationnonumber}

\begin{acknowledgments}
    The author wants to thank Bernhard Keller for his time to discuss the universal property of pretriangulated hulls.
\end{acknowledgments}

\section{Fourier-Mukai transforms for projective spaces}\label{section:products}

A famous theorem by Be{\u\i}linson \cite{Bei78} implies that the set of line bundles
$\{\OS(-i)\}_{i = 0,\dots,n}$ on the projective space $\Pro^n$ of dimension $n \in \N_0$ over $k$
generate the bounded derived category of coherent sheaves $\Db( \Pro^n )$ in a particularly nice way:
they form a so-called \emph{strong full exceptional collection} \cite[Chapter 8.3]{Huy_FM}.
Let $\BLineBundles{n}$ denote the full subcategory of $\Db( \Pro^n )$ spanned by this set of line bundles.
It follows that if we denote by 
$\Chb(\BLineBundles{n})$
the category of bounded complexes
whose objects consist of directs sums of bundles in $\BLineBundles{n}$,
and by $\Kb( \BLineBundles{n} )$ its homotopy category,
then the composition of the natural functors
\[
    \Kb( \BLineBundles{n} ) \hookrightarrow \Kb( \Pro^n ) \rightarrow \Db( \Pro^n )
\]
is an exact equivalence of triangulated categories, where $\Kb( \Pro^n)$ denotes the homotopy category of complexes of coherent sheaves on $\Pro^n$.

Whenever two smooth projective varieties have strong full exceptional collections of vector bundles,
forming external tensor products yields again a strong full exceptional collection for the product space \cite[Proposition 2.1.18]{Boehning06}.
This fact implies that
a product of projective spaces $\Pro^{\underline{n}} := \bigtimes_{i = 1}^l\Pro^{n_i}$
(for $l \in \N_0$ and $\underline{n} := (n_i)_i \in \N_0^l$)
admits a strong full exceptional collection given by the line bundles
$\boxtimes_{i = 1}^l \OS(s_i)$ for $s_i \in \{-n_i, \dots, 0\}$,
where the external tensor product $\boxtimes$ is defined
as 
\[
    \boxtimes_{i = 1}^l \OS(s_i) := \bigotimes_{i = 1}^l (p_i^{\ast} \OS(s_i))
\]
with $p_i$ denoting the $i$-th projection $\Pro^{\underline{n}} \longrightarrow \Pro^{n_i}$.
Again, we let $\BLineBundles{\underline{n}}$ denote the full subcategory of $\Db( \Pro^{\underline{n}} )$ 
spanned by this strong full exceptional collection.
Then the composition
\begin{equation}\label{equation:equivalence_kb_db}
    \iota := (\Kb( \BLineBundles{\underline{n}} ) \hookrightarrow \Kb( \Pro^{\underline{n}} ) \rightarrow \Db( \Pro^{\underline{n}} ))
\end{equation}
defines an exact equivalence.

In contrast to the classical construction of $\Db( \Pro^{\underline{n}} )$ as a categorical localization
at quasi-isomorphisms,
the triangulated category $\Kb( \BLineBundles{\underline{n}} )$ is constructed in far more explicit manner
and thus may serve as a constructive model for performing explicit computations.
Taking this computational point of view, it is natural to ask if prominent exact functors
 $\Db( \Pro^{\underline{n}} ) \rightarrow \Db( \Pro^{\underline{m}} )$ (for $j \in \N_0$, and $\underline{m} := (m_i)_i \in \N_0^j$),
regarded as functors $\Kb( \BLineBundles{\underline{n}} ) \rightarrow \Kb( \BLineBundles{\underline{m}} )$ via the equivalence \eqref{equation:equivalence_kb_db},
also admit a more explicit description.

The most prominent class of exact functors appearing in algebraic geometry are Fourier-Mukai transforms \cite[Chapter 5]{Huy_FM}.
\begin{definition}
    Let $X, Y$ be smooth projectives varieties over a field $k$.
    We denote by $\Db(X)$ the bounded derived category of $\Coh( X )$, the category of coherent sheaves on $X$.
    Let $K \in \Db( X \times Y )$ (the \emph{Fourier-Mukai kernel}).
    We define its induced \emph{Fourier-Mukai transform}
    \[
        \phi_K
        \coloneqq
        \RDer( {q}_{\ast} ) (K \otimesL p^{\ast}(-)):
        \Db( X ) \rightarrow \Db( Y )
    \]
    with $p: X \times Y \rightarrow X$ and $q: X \times Y \rightarrow Y$ the natural projections.
    An exact functor $F: \Db( X ) \rightarrow \Db(Y)$ is said to be of \emph{Fourier-Mukai type}
    if $F \simeq \phi_K$ for some $K  \in \Db( X \times Y )$.
\end{definition}
In other words, every Fourier-Mukai transform is the composition 
of three kinds of exact functors: a pullback functor of a projection, a tensor product, and a direct image functor of a projection
(all interpreted in a derived sense, with $p^{\ast}$ already being an exact functor between the corresponding categories of coherent sheaves). 
Thus, a computational modeling of Fourier-Mukai transforms can be reduced to the modeling of each of these three types of functors.
In a computational model of $\Db( \Pro^{\underline{n}} )$ relying on graded modules over the graded coordinate ring of $\Pro^{\underline{n}}$,
expressing the direct image functor explicitly is the most involved of the above three tasks \cite{EESProducts15}.
In contrast, we will see that using $\Kb( \BLineBundles{\underline{n}} )$ as a computational model,
the hardest task will be the modeling of the tensor product.

\subsection{A simple approach for modeling exact functors}\label{subsection:too_simple}

We describe a simple approach for the creation of exact functors
$\Kb( \BLineBundles{\underline{n}} ) \rightarrow \Kb( \BLineBundles{\underline{m}} )$.
Let $\BLineBundlesZero{\underline{n}}$ denote the full subcategory of $\Db( \BLineBundles{\underline{n}} )$
generated by the line bundles in $\BLineBundles{\underline{n}}$ together with the zero object 
(which we need to include in order to make an easy construction of the direct image functor possible in Subsection \ref{subsection:direct_image}).
Every functor
\[
    f: \BLineBundlesZero{\underline{n}} \longrightarrow \Chb(\BLineBundles{\underline{m}})
\]
(which only has to be described on morphisms between objects in $\BLineBundles{\underline{n}}$ since its extension to the zero object
is uniquely determined by the tacitly required $k$-linearity)
extends to a functor
\[
    F: \Chb(\BLineBundles{\underline{n}}) \longrightarrow \Chb(\BLineBundles{\underline{m}})
\]
by applying $f$ to every object of a given complex and taking the total complex of the resulting bicomplex.
Since taking total complexes and cones are functorial operations on the level of complexes, $F$ commutes with shifts and cones, and thus
gives rise to an exact functor
\[
    \overline{F}: \Kb(\BLineBundles{\underline{n}}) \longrightarrow \Kb(\BLineBundles{\underline{m}})
\]
between homotopy categories.
We may summarize this procedure within the following diagram (commutative up to natural isomorphism),
where $\Chb( \BLineBundles{\underline{n}} ) \rightarrow \Kb( \BLineBundles{\underline{n}} )$ denotes the natural quotient functor:
\begin{center}
    \begin{tikzpicture}[label/.style={postaction={
    decorate,
    decoration={markings, mark=at position .5 with \node #1;}},
    mylabel/.style={thick, draw=none, align=center, minimum width=0.5cm, minimum height=0.5cm,fill=white}}]
    \coordinate (r) at (5,0);
    \coordinate (d) at (0,-2);
    \node (L) {$\Chb(  \BLineBundles{\underline{n}} )$};
    
    \node (G) at ($(L) - (d)$){$\BLineBundlesZero{\underline{n}}$};
    
    \node (A) at ($(L) + (r)$) {$\Chb( \BLineBundles{\underline{m}} )$};
    
    \node (L2) at ($(L) + (d)$) {$\Kb(  \BLineBundles{\underline{n}} )$};
    \node (A2) at ($(A) + (d)$) {$\Kb( \BLineBundles{\underline{m}} )$};
    
    \node (L3) at ($(L2) + (d)$) {$\Db( \Pro^{\underline{n}} )$};
    \node (A3) at ($(A2) + (d)$) {$\Db( \Pro^{\underline{m}} )$};
    
    \draw[->,thick] (L) --node[above]{$F$} (A);
    \draw[->,thick] (L2) --node[above]{$\overline{F}$} (A2);
    \draw[->,thick] (L3) --node[above]{$\iota \circ \overline{F} \circ \iota^{-1}$} (A3);
    
    \draw[->,thick] (L) -- (L2);
    \draw[->,thick] (A) -- (A2);
    
    \draw[->,thick] (L2) --node[left]{$\iota$} (L3);
    \draw[->,thick] (A2) --node[right]{$\iota$} (A3);
    
    \draw[left hook->,thick] (G) -- (L);
    \draw[->,thick] (G) --node[above]{$f$} (A);
    
    \end{tikzpicture}
\end{center}

We will prove that for modeling the direct image and the pullback functors, this strategy can be carried out easily.

\subsection{Modeling the pullback functor}\label{subsection:pullback}
For simplicity, we restrict to the case $\Pro^a \times \Pro^b$ for $a, b \in \N_0$.
Computing within the small $k$-category $\BLineBundlesZero{a,b}$ is determined by the Künneth formula:
\[
    \Hom_{\Pro^a \times \Pro^b}( \OS( i_1) \boxtimes \OS( j_1), \OS( i_2) \boxtimes \OS( i_2 ) ) \simeq
    \Hom_{\Pro^a}( \OS( i_1) , \OS( i_2)  ) \otimes \Hom_{\Pro^b}( \OS( j_1) , \OS( j_2)  )
\]
for $i_1, i_2, j_1, j_2 \in \Z$.
Let $p: \Pro^a \times \Pro^b \rightarrow \Pro^a$ denote the first projection.
The pullback functor $p^{\ast}: \Db( \Pro^a ) \rightarrow \Db( \Pro^a \times \Pro^b)$
maps the subcategory $\BLineBundlesZero{a}$ into $\BLineBundlesZero{a,b}$, and yields the functor
\begin{equation*}
    f: \BLineBundlesZero{a} \rightarrow \BLineBundlesZero{a,b}:
    (\OS(i) \xrightarrow{\alpha} \OS(j))
    \mapsto
    (\OS(i) \boxtimes \OS \xrightarrow{\alpha \boxtimes id_\OS} \OS(j) \boxtimes \OS).
\end{equation*}
As described in Subsection \ref{subsection:too_simple} 
(applied to the composition $\BLineBundlesZero{a} \xrightarrow{f} \BLineBundlesZero{a,b} \hookrightarrow \Chb( \BLineBundles{a,b})$) 
we obtain an induced functor
\[
    \overline{F}: \Kb(\BLineBundles{a}) \rightarrow \Kb(\BLineBundles{a,b})
\]
between the homotopy categories.
\begin{lemma}
    Suppose given a bounded complex whose objects are direct sums of vector bundles in $\BLineBundles{a}$,
    then its pullback is given by replacing all direct summands of the form $\OS(i)$ with $\OS(i) \boxtimes \OS$.
    More precisely,
    we get a commutative diagram (up to natural isomorphism) of functors
    \begin{center}
        \begin{tikzpicture}[label/.style={postaction={
        decorate,
        decoration={markings, mark=at position .5 with \node #1;}},
        mylabel/.style={thick, draw=none, align=center, minimum width=0.5cm, minimum height=0.5cm,fill=white}}]
        \coordinate (r) at (5,0);
        \coordinate (d) at (0,-2);
        \node (L) {$\Kb(  \BLineBundles{a} )$};
        \node (A) at ($(L) + (r)$) {$\Kb( \BLineBundles{a,b} )$};
        
        \node (L2) at ($(L) + (d)$) {$\Db( \Pro^a) $};
        \node (A2) at ($(A) + (d)$) {$\Db( \Pro^{a} \times \Pro^{b}) $};
        
        \draw[->,thick] (L) --node[above]{$\overline{F}$} (A);
        \draw[->,thick] (L2) --node[above]{$p^{\ast}$} (A2);
        
        \draw[->,thick] (L) --node[left]{$\iota$} (L2);
        \draw[->,thick] (A) --node[right]{$\iota$} (A2);
        
        \end{tikzpicture}
    \end{center}
\end{lemma}
\begin{proof}
This is true since $p^{\ast}$ seen as a functor between abelian categories of coherent sheaves is an exact functor,
and thus its derived version is given by its application to complexes.
\end{proof}

\subsection{Modeling the direct image functor}\label{subsection:direct_image}

Next, we can define a right adjoint of $f$ without any reference to algebraic geometry. It has to be given by 
\begin{equation*}
    g: \BLineBundlesZero{a,b} \rightarrow \BLineBundlesZero{a}:
    \OS(i) \boxtimes \OS(j) \mapsto
    \left\{ \begin{array}{cc}
        0, & j < 0 \\
        \OS(i), & i = 0
\end{array} \right.
\end{equation*}
From this definition, it becomes clear why we had to add the zero objects to $\BLineBundlesZero{a}$ and $\BLineBundlesZero{a,b}$,
otherwise, we could not have defined this adjoint.
Indeed, $g$ is right adjoint to $f$ since
\begin{equation*}
    \Hom_{\BLineBundlesZero{a,b}}( f( \OS(i_1)), \OS(i_2) \boxtimes \OS(j_2) ) \simeq \Hom_{\BLineBundlesZero{a,b}}( \OS(i_1) \boxtimes \OS, \OS(i_2) \boxtimes \OS(j_2) )
\end{equation*}
is zero in the case $j_2 < 0$, since $\Hom( \OS, \OS( j_2) ) = 0$, and equivalent to
\begin{equation*}
    \Hom_{\BLineBundlesZero{a}}( \OS(i_1), \OS(j_2) ) \simeq \Hom_{\BLineBundlesZero{a}}( \OS(i_1) , g(\OS(i_2) \boxtimes \OS(j_2)) )
\end{equation*}
for $j_2 = 0$.
Again, $g$ induces a functor
\[
    \overline{G}: \Kb( \BLineBundles{a,b} ) \rightarrow \Kb( \BLineBundles{a} )
\]
between homotopy categories.

\begin{lemma}
    The functor $\overline{G}$ is right adjoint to $\overline{F}$.
\end{lemma}
\begin{proof}
    The proof can be done by pure category theory.
    The adjunction between $g$ and $f$
    first gives rise to an adjunction between their additive closures,
    second to the category of complexes,
    and third passes down to homotopy since the induced functors on the level of complexes respect homotopy.
\end{proof}

\begin{corollary}
    Suppose given a bounded complex whose objects are direct sums of vector bundles in $\BLineBundles{a,b}$,
    then its direct image is given by deleting all direct summands of the form $\OS(i) \boxtimes \OS(j)$ with $j < 0$,
    and replacing the remaining bundles $\OS(i) \boxtimes \OS$ with $\OS(i)$.
    More precisely,
    we get a commutative diagram (up to natural isomorphism) of functors
    \begin{center}
        \begin{tikzpicture}[label/.style={postaction={
        decorate,
        decoration={markings, mark=at position .5 with \node #1;}},
        mylabel/.style={thick, draw=none, align=center, minimum width=0.5cm, minimum height=0.5cm,fill=white}}]
        \coordinate (r) at (5,0);
        \coordinate (d) at (0,-2);
        \node (L) {$\Kb(  \BLineBundles{a,b} )$};
        \node (A) at ($(L) + (r)$) {$\Kb( \BLineBundles{a} )$};
        
        \node (L2) at ($(L) + (d)$) {$\Db( \Pro^{a} \times \Pro^{b})$};
        \node (A2) at ($(A) + (d)$) {$\Db( \Pro^a)$};
        
        \draw[->,thick] (L) --node[above]{$\overline{G}$} (A);
        \draw[->,thick] (L2) --node[above]{$\RDer p_{\ast}$} (A2);
        
        \draw[->,thick] (L) --node[left]{$\iota$} (L2);
        \draw[->,thick] (A) --node[right]{$\iota$} (A2);
        
        \end{tikzpicture}
    \end{center}
\end{corollary}
\begin{proof}
    $\RDer p_{\ast}$ is characterized up to natural isomorphism as the right adjoint of $p^{\ast}$.
\end{proof}

\subsection{The difficulty to model tensor products}\label{subsection:difficulty}
As a concrete example, we discuss the functor
\[
    (\OS(-1) \otimesL - ): \Db( \Pro^2 ) \rightarrow \Db( \Pro^2 ).
\]
The category $\BLineBundles{2}$ can be presented by the Beilinson quiver
\vspace{1em}
\begin{center}
    \begin{tikzpicture}[label/.style={postaction={
    decorate,
    decoration={markings, mark=at position .5 with \node #1;}},
    mylabel/.style={thick, draw=none, align=center, minimum width=0.5cm, minimum height=0.5cm,fill=white}}]
    \coordinate (r) at (4,0);
    \node (A) {$\OS(-2)$};
    \node (B) at ($(A)+(r)$) {$\OS(-1)$};
    \node (C) at ($(B) + (r)$){$\OS$};
    \draw[->,thick,label={[mylabel]{$y_0$}},transform canvas={yshift=1em}] (A) -- (B);
    \draw[->,thick,label={[mylabel]{$y_1$}},transform canvas={yshift=0em}] (A) -- (B);
    \draw[->,thick,label={[mylabel]{$y_2$}},transform canvas={yshift=-1em}] (A) -- (B);
    \draw[->,thick,label={[mylabel]{$x_0$}},transform canvas={yshift=1em}] (B) -- (C);
    \draw[->,thick,label={[mylabel]{$x_1$}},transform canvas={yshift=0em}] (B) -- (C);
    \draw[->,thick,label={[mylabel]{$x_2$}},transform canvas={yshift=-1em}] (B) -- (C);
    \end{tikzpicture}
\end{center}
with relations given by
\[
 x_i \circ y_j = x_j \circ y_i \hspace{2em} 0 \leq i,j \leq 2.
\]

In $\Db( \Pro^2 )$, we have an isomorphism $\OS( -3 ) \cong \mathcal{C}$ due to the Koszul complex,
where $\mathcal{C}$ is the complex situated in degrees $0,1,2$:
\begin{center}
    \begin{tikzpicture}[label/.style={postaction={
    decorate,
    decoration={markings, mark=at position .5 with \node #1;}},
    mylabel/.style={thick, draw=none, align=center, minimum width=0.5cm, minimum height=0.5cm,fill=white}}]
    \coordinate (r) at (4,0);
    \node (A) {$\OS(-2)^{3 \times 1}$};
    \node (B) at ($(A)+(r)$) {$\OS(-1)^{3 \times 1}$};
    \node (C) at ($(B) + (r)$){$\OS$,};
    \draw[->,thick,label={[above]{$\rho$}}] (A) -- (B);
    \draw[->,thick,label={[above]{$\delta$}}] (B) -- (C);
    \end{tikzpicture}
\end{center}
with
\[
 \rho :=
 \begin{pmatrix}
 -y_1 & y_2 & 0 \\
  y_0 & 0 & -y_2 \\
  0 & -y_0 & y_1 \\
 \end{pmatrix}
\]
and 
\[
 \delta :=
 \begin{pmatrix}
  x_0 &
  x_1 &
  x_2
 \end{pmatrix}.
\]
We will show that there is no functor
\[
    f: \BLineBundlesZero{2} \longrightarrow \Chb( \BLineBundles{2} )
\]
whose values on objects are given by
\begin{align*}
    f( \OS ) &= \OS(-1), \\
    f( \OS(-1) ) &= \OS(-2), \\
    f( \OS(-2) ) &= \mathcal{C}, \\
   \end{align*}
that fits into a commutative diagram (up to natural isomorphism)
\begin{center}
    \begin{tikzpicture}[label/.style={postaction={
    decorate,
    decoration={markings, mark=at position .5 with \node #1;}},
    mylabel/.style={thick, draw=none, align=center, minimum width=0.5cm, minimum height=0.5cm,fill=white}}]
    \coordinate (r) at (5,0);
    \coordinate (d) at (0,-2);
    \node (L) {$\BLineBundlesZero{2}$};
    \node (A) at ($(L) + (r)$) {$\Chb( \BLineBundles{2} )$};
    
    \node (L2) at ($(L) + (d)$) {$\Db( \Pro^{2} )$};
    \node (A2) at ($(A) + (d)$) {$\Db( \Pro^2)$.};
    
    \draw[->,thick] (L) --node[above]{$f$} (A);
    \draw[->,thick] (L2) --node[above]{$(\OS(-1) \otimesL - )$} (A2);
    
    \draw[->,thick] (L) -- (L2);
    \draw[->,thick] (A) -- (A2);
    
    \end{tikzpicture}
\end{center}
We refrain from studying the question whether there exists such an $f$ with different but homotopic values on objects,
since this example is for motivational purposes only and will later be seen to be rectified by the machinery of \Ainf-functors.

Assume there were such an $f$. On morphisms, it has to be given by
\[
    f( x_i ) = y_i
\]
up to non-zero scalars for $i = 0,1,2$, since $y_i = \OS(-1) \otimes x_i$.
Next, we remark that 
\[
    \Hom_{\Chb( \BLineBundles{2} )}( \mathcal{C}, \OS(-2)) \simeq \Hom_{\BLineBundlesZero{2}}( \OS(-2)^{3 \times 1}, \OS(-2)) \simeq k^{3 \times 1}.
\]
Since $f$ is faithful, the $f(y_i)$'s form a basis of $\Hom_{\Chb( \BLineBundles{2} )}( \mathcal{C}, \OS(-2))$,
which we denote by $z_i$.
But now, the elements $y_i \circ z_j \in \Hom_{\Chb( \BLineBundles{2} )}( \mathcal{C}, \OS(-1)) \simeq \Hom_{\BLineBundlesZero{2}}( \OS(-2), \OS(-1))^{3 \times 1}$ 
are linearly independent for all $i,j = 0,1,2$.
In particular, 
\[f( x_1 \circ y_0 ) \neq f( x_0 \circ y_1 ),\]
a contradiction.

\section{Fourier-Mukai transforms as \Ainf-functors between dg categories}\label{section:dg_ainf_theory}

In this section, we will present the machinery of dg enhancements and \Ainf-functors
from an abstract point of view
in order to generalize and strengthen our construction strategy of Subsection \ref{subsection:too_simple}.
From this machinery, it will follow that 
regarding the concrete situation of Section \ref{section:products}, there exist dg categories
$\Chbdg( \BLineBundles{\underline{m}} )$ and $\Chbdg(  \BLineBundles{\underline{n}} )$
whose homotopy categories are exactly
$\Kb(  \BLineBundles{\underline{m}} )$
and
$\Kb(  \BLineBundles{\underline{n}} )$,
and any \Ainf-functor $f: \BLineBundles{\underline{n}} \rightarrow \Chbdg( \BLineBundles{\underline{m}} )$
gives rise to a diagram
\begin{center}
    \begin{tikzpicture}[label/.style={postaction={
    decorate,
    decoration={markings, mark=at position .5 with \node #1;}},
    mylabel/.style={thick, draw=none, align=center, minimum width=0.5cm, minimum height=0.5cm,fill=white}}]
    \coordinate (r) at (5,0);
    \coordinate (d) at (0,-2);
    \node (L) {$\Chbdg(  \BLineBundles{\underline{n}} )$};
    
    \node (G) at ($(L) - (d)$){$\BLineBundles{\underline{n}}$};
    
    \node (A) at ($(L) + (r)$) {$\Chbdg( \BLineBundles{\underline{m}} )$};
    
    \node (L2) at ($(L) + (d)$) {$\Kb(  \BLineBundles{\underline{n}} )$};
    \node (A2) at ($(A) + (d)$) {$\Kb( \BLineBundles{\underline{m}} )$};
    
    \node (L3) at ($(L2) + (d)$) {$\Db( \Pro^{\underline{n}} )$};
    \node (A3) at ($(A2) + (d)$) {$\Db( \Pro^{\underline{m}} )$};
    
    \draw[->,thick] (L) --node[above]{$F$} (A);
    \draw[->,thick] (L2) --node[above]{$\overline{F}$} (A2);
    \draw[->,thick] (L3) --node[above]{$\iota \circ \overline{F} \circ \iota^{-1}$} (A3);
    
    \draw[draw=none] (L) --node[rotate=-90]{$\longmapsto$} node[left]{$\Hzero$} (L2);
    \draw[draw=none] (A) --node[rotate=-90]{$\longmapsto$} node[right]{$\Hzero$} (A2);
    
    \draw[->,thick] (L2) --node[left]{$\iota$} (L3);
    \draw[->,thick] (A2) --node[right]{$\iota$} (A3);
    
    \draw[left hook->,thick] (G) -- (L);
    \draw[->,thick] (G) --node[above]{$f$} (A);
    
    \end{tikzpicture}
\end{center}
whose top triangle consists of \Ainf-functors and commutes, 
and whose lower square consists of exact functors which commute up to natural isomorphism.
This machinery will later lead to a solution of the problem posed in Subsection \ref{subsection:difficulty},
see Subsection \ref{subsection:quivers} for the solution.

Now, we will discuss the theoretical notions relevant for this construction strategy.

\subsection{Enhancements of triangulated categories}

In this subsection, we follow closely the nice exposition of the material given in \cite[Section 1]{CSdg17}.
A dg category is a category enriched over complexes over $k$,
and a dg functor is a corresponding enriched functor
(see Definition \ref{definition:dg_cat} and Definition \ref{definition:dg_fun} for an expansion of these definitions).
Small dg categories and dg functors form a category $\dgCat$.
It is a closed symmetric monoidal category \cite[Section 2.3]{Kel06} with tensor product
$\AC \otimes \BC$, whose objects are pairs of objects in $\AC$ and $\BC$,
and whose morphism complexes arise from the tensor product of complexes.
The objects in the dg category of internal homomorphisms $\IHom( \AC, \BC )$
are given by the dg functors.

The \emph{opposite} of a dg category $\AC$, denoted by $\AC^{\opC}$,
has the same objects as $\AC$, morphism complexes $( A_1, A_2 )_{\AC^{\opC}} := ( A_2, A_1 )_{\AC}$,
and composition $\alpha \circ_{(\AC^{\opC})} \beta := (-1)^{\abs{\alpha}\abs{\beta}} \beta \circ_{\AC} \alpha$ for composable morphisms $\alpha, \beta$ in $\AC$.
The \emph{underlying category}, resp. \emph{homotopy category}, of $\AC$ is denoted by $Z^0(\AC)$, resp. $\Hzero(\AC)$,
it has the same objects as $\AC$,
morphisms are given by the $0$-th cocycles $Z^0( ( A_1, A_2) )$, resp. $0$-th cohomologies $\Hzero( ( A_1, A_2) )$,
for all $A_1, A_2 \in \AC$.
Any dg functor $F$ induces a functor $Z^0F$, resp. $\Hzero F$, on the underlying category, resp. homotopy category.

\begin{example}\label{example:ch_dg}
    Let $\AC$ be an additive category (over $k$).
    We can form a dg category $\Ch_{\dg}(\AC)$ whose objects are given by complexes in $\AC$.
    Its degree $g \in \Z$ morphisms are
    $( A_1, A_2 )^g_{\Ch_{\dg}(\AC)} := \prod_{i \in \Z} ( A_1^{i}, A_2^{i + g})_{\AC}$,
    the differential is defined by the formula $d(\alpha) := d_{A_2} \circ \alpha - (-1)^{g}\alpha \circ d_{A_1}$
    for an $\alpha$ of degree $g$, where the application of the differentials $d_{A_1}$, $d_{A_2}$ is performed componentwise.
    This dg category is designed in a way such that $Z^0( \Ch_{\dg}(\AC) )$ is the category of complexes $\Ch( \AC )$,
    and $\Hzero( \Ch_{\dg}(\AC) )$ is the category of complexes up to homotopy $\KCh( \AC )$.
    The full dg subcategory of $\Ch_{\dg}( \AC )$ generated by bounded complexes is denoted by $\Chbdg(\AC)$.
\end{example}

The dg category of right dg modules is defined as $\Modr\AC \coloneqq \IHom( \AC^{\opC}, \Ch_{\dg}(k\Modl) )$.
Its homotopy category is known to be a triangulated category in a natural way. Moreover, we have a dg functor
\[
    Y^{\AC}: \AC \longrightarrow \Modr\AC: A \mapsto ( -, A )_{\AC}
\]
called \emph{Yoneda functor}.

\begin{definition}\label{definition:pretriangulated}
    A dg category $\AC$ is called \emph{pretriangulated}
    if the essential image of
    \[
    \Hzero(Y^{\AC}): \Hzero(\AC) \longrightarrow \Hzero(\Modr\AC)
    \]
    is a triangulated subcategory.
\end{definition}

\begin{remark}
    Thus, being pretriangulated is a property of a dg category.
    Indeed, being closed under cones and shifts can be characterized intrinsically by
    the dg representability of appropriate dg functors.
\end{remark}

\begin{remark}\label{remark:plus_tria}
    The notion of a pretriangulated category goes back to Bondal and Kapranov \cite{BonKap90},
    but they used it in a stronger sense: instead of cones, they considered arbitrary convolutions.
    In the same paper, they used the term $+$-pretriangulated for the concept of Definition \ref{definition:pretriangulated}.
\end{remark}

\begin{definition}
    Let $\TC$ be a triangulated category.
    A dg enhancement of $\TC$ is a pretriangulated dg category $\AC$ together with an exact equivalence
    $\phi: \Hzero( \AC ) \xrightarrow{\sim} \TC$.
\end{definition}

\begin{example}\label{example:chb_enh}
    The dg category $\Chbdg( \AC )$ 
    of Example \ref{example:ch_dg}
    is a dg enhancement of $\Kb( \AC )$.
    In particular, the category $\Kb(  \BLineBundles{\underline{n}} )$ of Section \ref{section:products}
    has a dg enhancement given by $\Chbdg(  \BLineBundles{\underline{n}} )$.
\end{example}

\begin{example}{\cite[Section 5]{BLL04}}\label{example:dg_enh_db}
    Let $X$ be a smooth projective variety.
    Let $\OS_X\Modl$ be the category of $\OS_X$-modules.
    The full dg subcategory of $\Ch_{\dg}( \OS_X\Modl )$
    generated by complexes that
    \begin{itemize}
        \item are bounded below,
        \item consist only of injective objects,
        \item have bounded cohomology,
        \item have coherent cohomology objects,
    \end{itemize}
    gives rise to a dg enhancement $\Db_{\dg}(X)$ of $\Db( X )$.
    Note that in the case $X = \Pro^n$, $\Chbdg(  \BLineBundles{\underline{n}} )$ also provides a (much smaller) enhancement of
    $\Db( \Pro^n )$ by Example \ref{example:chb_enh}.
\end{example}

\subsection{The homotopy category of dg categories}

Let $\AC$, $\BC$ be additive categories.
Every dg functor $F: \Chbdg( \AC ) \rightarrow \Chbdg( \BC )$
gives rise to a functor $Z^0F: \Chb( \AC ) \rightarrow \Chb( \BC )$
compatible with homotopies, which in turn gives rise to a functor
$\Hzero F: \Kb( \AC ) \rightarrow \Kb( \BC )$.
In this sense, the notion of a dg functor appears to be even less flexible 
for the purpose of modeling functors $\Kb( \AC ) \rightarrow \Kb( \BC )$
in comparison with the approach described in Subsection \ref{subsection:too_simple}.
Luckily, dg functors can be replaced by a far more flexible notion which is better adapted to functors between homotopy categories.

\begin{definition}
    Let $\AC$, $\BC$ be dg categories.
    A dg functor $F: \AC \rightarrow \BC$
    is called a \emph{quasi-equivalence} if
    \begin{itemize}
        \item $( A_1, A_2 )_{\AC} \rightarrow ( F(A_1), F(A_2) )_{\BC}$
        are quasi-isomorphisms for all $A_1, A_2 \in \AC$,
        \item $\Hzero(F): \Hzero(\AC) \rightarrow \Hzero(\BC)$ is an equivalence of categories.
    \end{itemize}
\end{definition}

The sought-after more flexible notion is exactly provided by the idea of making quasi-equivalences invertible in $\dgCat$.
In order to achieve this goal,
one uses the theory of model structures introduced by Quillen \cite{Quillen67}.
There exists a model structure on $\dgCat$ whose weak equivalences are given by quasi-equivalences \cite{Tab05}.
Thus, the homotopy category of $\dgCat$ w.r.t.\ this model structure, denoted by
$\Hqe$, is the localization of $\dgCat$ at quasi-equivalences.

Let $\HCat$ denote the category of small categories with functors considered up to natural isomorphisms as morphisms.
There is a functor (see \cite[Section 2]{Toen07})
\[
    \Hzero: \Hqe \rightarrow \HCat: (\AC \xrightarrow{F}{\BC}) \mapsto (\Hzero(\AC) \xrightarrow{\Hzero(F)}{\Hzero(\BC)})
\]
which means that every morphism in $\Hqe$ gives rise to a functor (up to natural isomorphism) between homotopy categories.
Note that this functor extends the canonical functor $\Hzero: \dgCat \rightarrow \HCat$
and thus extends the way in which we can create morphisms in $\HCat$.
It is thus of interest to understand the homomorphism sets in $\Hqe$.

The category $\Hqe$ inherits its tensor product $\otimesL$ as a derivation from the tensor product in $\dgCat$ \cite[Section 4]{Toen07}.
Its a theorem due to Toën that $\otimesL$ gives rise to a closed monoidal structure on $\Hqe$ \cite[Theorem 1.3]{Toen07},
and we denote the internal homomorphisms by $\IHomR$.

For dg categories $\AC, \BC$, the isomorphism classes in
$\Hzero(\IHomR( \AC, \BC ))$ are in bijection with the elements in $\Hom_{\Hqe}( \AC, \BC )$ \cite{Toen07}.
A concrete description of $\IHomR( \AC, \BC )$ can be given
in the language of \Ainf-categories.
An \Ainf-category (with strict units) can be regarded as a kind of category whose composition is associative only up to 
concretely given homotopies, which in turn have to satisfy compatibilities up to even higher homotopies, and so on
(see Definition \ref{definition:a_inf_cat} for the details).
The theory of \Ainf-algebras (which can be seen as \Ainf-categories with only one object)
goes back to Stasheff \cite{Stasheff63}, see \cite{KelAinf06} for a survey.
For our purposes, we only need to know that dg categories can be regarded as a special case of \Ainf-categories
(see Construction \ref{construction:dg_to_Ainf}).
In particular, it makes sense to speak of an \Ainf-functor (which respects strict units, see Definition \ref{definition:a_inf_func})
between dg categories, and this yields a vast relaxation of the notion of a dg functor.
In fact, 
$\IHomR( \AC, \BC )$ can be seen
as the dg category of \Ainf-functors from $\AC$ to $\BC$,
meaning that the elements in $\Hom_{\Hqe}( \AC, \BC )$
may be regarded as isomorphism classes of \Ainf-functors,
and $\Hzero: \Hqe \rightarrow \HCat$ as the functor mapping an \Ainf-functor to its induced functor on homotopy categories,
see \cite[Section 4.3]{Kel06}, or more recently \cite{COSLoc18}.

\subsection{Generating pretriangulated categories}

Let $\dgCat_{\pretr}$ denote the full subcategory of $\dgCat$ generated by pretriangulated dg categories.
There is a functor
\[
    \pretr: \dgCat \rightarrow \dgCat_{\pretr}
\]
and an isomorphism
\[
    \Hom_{\dgCat}( \AC, \BC ) \simeq \Hom_{\dgCat_{\pretr}}( \pretr(\AC), \BC ) 
\]
natural in every dg category $\AC$ and every pretriangulated dg category $\BC$,
constructed in \cite[Section 4]{BonKap90} in the language of monads and under the name ``+-pretriangulated'' (Remark \ref{remark:plus_tria}),
see also \cite[Section 4.5]{Kel06}.
In other words, $\dgCat_{\pretr}$ is a reflective subcategory of $\dgCat$.
The dg category $\pretr(\AC)$ is called the \emph{pretriangulated hull} of $\AC$.
Pretriangulated hulls together with their natural embedding $\AC \rightarrow \pretr( \AC )$
can be described very explicitly \cite{BonKap90} in terms of one-sided twisted complexes (see Subsection \ref{subsubsection:pretr}).

\begin{example}
    In the case when $\AC$ is a $k$-category regarded as a dg category,
    $\pretr( \AC )$ is equivalent to $\Chbdg( \AC )$.
\end{example}

The functor $\pretr$ respects quasi-equivalences \cite[Section 4.5]{Kel06}, thus, the above
adjunction passes down to the level of homotopy
\[
    \pretr: \Hqe \rightarrow \Hqe_{\pretr}
\]
where $\Hqe_{\pretr}$ denotes the full subcategory of $\Hqe$ generated by pretriangulated dg categories
and we have an isomorphism
\begin{equation}\label{equation:pretre_adj}
    \Hom_{\Hqe}( \AC, \BC ) \simeq \Hom_{\Hqe_{\pretr}}( \pretr(\AC), \BC )
\end{equation}
natural in every dg category $\AC$ and every pretriangulated dg category $\BC$ \cite[Section 4.5]{Kel06}.
Since the adjunction on the homotopy level is induced by the one on the dg level,
the direction $\Hom_{\Hqe_{\pretr}}( \pretr(\AC), \BC ) \rightarrow \Hom_{\Hqe}( \AC, \BC )$
of the isomorphism is induced by the composition with the natural embedding $\AC \rightarrow \pretr( \AC )$.
One of the main technical goals of this paper is to describe the other direction of this isomorphism explicitly
(see Theorem \ref{theorem:lift_theorem}).

The functor $\pretr$ is extremely useful for the explicit construction of triangulated categories.
For every pretriangulated dg category $\BC$, the natural dg functor 
$\pretr( \BC ) \rightarrow \BC$ (arising from the counit of the adjunction above evaluated at $\BC$) is an isomorphism, 
and thus $\BC$ can be seen as being ``freely generated by itself''.
If $\AC \xrightarrow{F} \BC$ is a full and faithful dg functor, then the essential image of
\[
    \Hzero(\pretr( \AC )) \rightarrow \Hzero(\pretr( \BC )) \xrightarrow{\sim} \Hzero(\BC)
\]
is the triangulated subcategory of $\Hzero(\BC)$
generated by the objects in the image of $F$ (see, e.g., \cite{BonKap90}).

\begin{example}\label{example:economic_dg}
    Let $X$ be a smooth projective variety.
    Let $\mathbb{G} \subset \Obj(\Db(X))$ generate $\Db(X)$ as a triangulated category.
    Let $\GC$ be a dg category
    which is isomorphic in $\Hqe$ to
    the full dg subcategory 
    of $\Db_{\dg}(X)$ 
    spanned by $\mathbb{G}$ (or by isomorphic representatives).
    Then we get an exact equivalence
    \[
        \Hzero(\pretr (\GC)) \simeq \Db(X).
    \]
    If furthermore $\mathbb{G}$ is a full exceptional collection,
    then the dg category $\GC$ can be chosen in way such that it has finite dimensional homomorphism spaces \cite[Theorem 1.1]{Bod15}.
    A particular situation occurs if the full exceptional collection is strong. 
    In that case, $\GC$ can be chosen as the full $k$-subcategory of $\Db(X)$ generated by $\mathbb{G}$ \cite[Remark 3.1]{Bod15}.
\end{example}

\subsection{Inducing Fourier-Mukai transforms}\label{subsection:inducing_fm}
Let $X$ be a smooth projective variety over $k$.
The dg enhancement $\Dbdg(X)$ of $\Db( X )$ given in Example \ref{example:dg_enh_db}
is unique up to isomorphism in $\Hqe$.
More precisely, let $\TC$ be a triangulated category with a dg enhancement.
We say that $\TC$ has a \emph{strongly unique} dg enhancement
if for any two enhancements $(\AC, \phi: \Hzero(\AC) \xrightarrow{\sim} \TC )$
and $(\BC, \psi: \Hzero(\BC) \xrightarrow{\sim} \TC )$,
there exists an isomorphism $f \in \Hom_{\Hqe}( \AC, \BC )$
such that there is a natural isomorphism
\[
    \psi \circ \Hzero(f) \simeq \phi.
\]
By \cite[Theorem 5.14]{COSLoc18}, the category $\Db(X)$ has a strongly unique dg enhancement.
In particular, one may replace $\Db_{\dg}( X )$ and $\Db_{\dg}( Y )$ in the following theorem
with the computationally more suitable enhancements of Example \ref{example:economic_dg}.
\begin{theorem}[{\cite[Proposition 6.11]{COSLoc18}}]\label{theorem:FM_as_Ainf}
    Let $X, Y$ be smooth projective varieties over a field $k$.
    The image of the map
    \[
        \Hzero: \Hom_{\Hqe}( \Db_{\dg}( X ), \Db_{\dg}( Y ) ) \rightarrow \Hom_{\HCat}( \Db( X ), \Db( Y ) )
    \]
    consists exactly of the functors of Fourier-Mukai type.
\end{theorem}

Assume that $\Db(X)$ and $\Db(Y)$ have dg enhancements $\pretr( \GC_X )$ and $\pretr( \GC_Y )$
for dg categories $\GC_X$, $\GC_Y$.
Then Theorem \ref{theorem:FM_as_Ainf} implies that any 
Fourier-Mukai transform arises from an \Ainf-functor
$G: \pretr(\GC_X ) \rightarrow \pretr(\GC_Y)$
by passing to homotopy categories.
If we restrict $G$ to $\GC_X$ and denote this restriction by
\[
    F: \GC_X \rightarrow \pretr(\GC_Y ),
\]
then we may reconstruct $G$ by applying the adjunction \eqref{equation:pretre_adj}, i.e.,
$G = F^{\sharp}$ in $\Hqe$ with $F^{\sharp}$ the \Ainf-functor induced by the universal property of pretriangulated hulls.
We summarize this procedure within the following diagram,
where the resulting functor $\phi$ is a Fourier-Mukai transform:
\begin{center}
    \begin{tikzpicture}[label/.style={postaction={
    decorate,
    decoration={markings, mark=at position .5 with \node #1;}},
    mylabel/.style={thick, draw=none, align=center, minimum width=0.5cm, minimum height=0.5cm,fill=white}}]
    \coordinate (r) at (5,0);
    \coordinate (d) at (0,-2);
    \node (L) {$\pretr(\GC_{X})$};
    
    \node (G) at ($(L) - (d)$){$\GC_{X}$};
    
    \node (A) at ($(L) + (r)$) {$\pretr(\GC_{Y})$};
    
    \node (L2) at ($(L) + (d)$) {$\Hzero(\pretr(\GC_{X}))$};
    \node (A2) at ($(A) + (d)$) {$\Hzero(\pretr(\GC_{Y}))$};
    
    \node (L3) at ($(L2) + (d)$) {$\Db( X )$};
    \node (A3) at ($(A2) + (d)$) {$\Db( Y )$};
    
    \draw[->,thick] (L) --node[above]{$F^{\sharp}$} (A);
    \draw[->,thick] (L2) --node[above]{$\Hzero({F^{\sharp}})$} (A2);
    \draw[->,thick] (L3) --node[above]{$\phi$} (A3);
    
    \draw[draw=none] (L) --node[rotate=-90]{$\longmapsto$} node[left]{$\Hzero$} (L2);
    \draw[draw=none] (A) --node[rotate=-90]{$\longmapsto$} node[right]{$\Hzero$} (A2);
    
    \draw[->,thick] (L2) --node[above,rotate=90]{$\sim$} (L3);
    \draw[->,thick] (A2) --node[above,rotate=-90]{$\sim$} (A3);
    
    \draw[left hook->,thick] (G) -- (L);
    \draw[->,thick] (G) --node[above]{$F$} (A);
    
    \end{tikzpicture}
\end{center}
In the next Section \ref{section:constructive}, 
we answer the following question:
suppose given an \Ainf-functor $F: \GC_X \rightarrow \pretr(\GC_Y)$, how do we compute explicitly the corresponding lift $F^{\sharp}$?
In Subsection \ref{subsection:fm_for_proj}, we will resume our modeling of Fourier-Mukai transforms for projective spaces,
building on the ideas presented in this section.

\section{A constructive approach to \Ainf-functors between dg categories}\label{section:constructive}

This section is the technical heart of the paper. It can be seen as an
implementation strategy for \Ainf-functors between dg categories in a software project like \CapPkg
\cite{CAP-project, GP_Rundbrief}.

\subsection{Formal constructions of dg categories}

\subsubsection{Preliminaries}

We recall the expanded definitions of dg categories and dg functors.

\begin{definition}\label{definition:dg_cat}
    A \emph{dg category} $\AC$ consists of the following data:
    \begin{enumerate}
        \item A collection of objects, denoted by $\Obj_{\AC}$.
        \item For all objects $A_1, A_2$, a graded space $(A_1, A_2)$, called \emph{morphisms from $A_1$ to $A_2$}.
        \item For all objects $A_1, A_2$, a graded map \[d: (A_1, A_2) \longrightarrow (A_1, A_2)\] of degree $1$, called \emph{differential}.
        \item For all objects $A_1, A_2, A_3$, a graded map \[m_2: (A_2, A_3) \otimes (A_1, A_2) \longrightarrow (A_1, A_3)\] of degree $0$.
              We also use infix notation $\alpha \circ \beta := m_2( \alpha, \beta )$ and call this operation \emph{(post)composition}.
        \item The following equations hold ($\alpha, \beta, \gamma$ are composable morphisms):
        \begin{itemize}
            \item (differential equation): $d \circ d = 0$,
            \item (graded Leibniz rule): $d( \alpha \circ \beta ) = d( \alpha ) \circ \beta + (-1)^{\abs{\alpha}}\alpha \circ d(\beta)$,
            \item (associativity): $\alpha \circ (\beta \circ \gamma ) = (\alpha \circ \beta) \circ \gamma$.
        \end{itemize}
        \item For all objects $A$, a morphism $\id_A \in (A,A)$ of degree $0$, called \emph{identity}.
        \item For all objects $A_1, A_2$ and $\alpha \in (A_1, A_2)$, we have
        \begin{itemize}
            \item $\alpha \circ \id_{A_1} = \alpha$,
            \item $\id_{A_2} \circ \alpha = \alpha$,
            \item $d( \id_{A_1} ) = 0$.
        \end{itemize}
    \end{enumerate}
    \end{definition}

\begin{definition}
    A morphism $\alpha$ in a dg category $\AC$ is called \emph{closed}
    if $d( \alpha ) = 0$.
\end{definition}
    
\begin{construction}\label{construction:k_linear_to_dg}
    A $k$-category $\AC$ gives rise to a dg category $\AC_{\dg}$ with the same objects as $\AC$, and
    the homomorphism spaces can be considered as graded spaces concentrated in degree $0$,
    the differential is $d := 0$.
    Sometimes, we tacitly use this conversion and omit the subscript $(-)_{\dg}$.
\end{construction}

\begin{definition}\label{definition:dg_fun}
    Given dg categories $\AC$, $\BC$,
    a \emph{dg functor} $F: \AC \rightarrow \BC$ consists of the following data:
    \begin{enumerate}
     \item A map $F: \Obj_{\AC} \rightarrow \Obj_{\BC}$.
     \item For $A_1, A_2 \in \AC$, a morphism of graded spaces
           \[
            F_{A_1,A_2}: ( A_1, A_2 ) \longrightarrow ( FA_1, FA_2 ),
           \]
           which is also simply denoted by $F$,
           and which is compatible with the differential, composition, and identities.
    \end{enumerate}
   \end{definition}

\subsubsection{Completion by direct sums}\label{subsubsection:direct_sums}

In \cite[Tag 09L4]{stacks-project},
a differential graded direct sum for two objects in
a dg category is introduced.
We are going to call such objects simply direct sums.

\begin{definition}
 Let $\AC$ be a dg category.
 Given finitely many objects $A_1, \dots, A_n \in \AC$ for $n \in \N_0$,
 a \emph{direct sum} consists of the following data:
 \begin{enumerate}
  \item An object $\bigoplus_{i=1}^n A_i$ in $\AC$.
  \item \emph{Closed degree $0$ morphisms} $\iota_i: A_i \rightarrow \bigoplus_{i=1}^n A_i$ and $\pi_i: \bigoplus_{i=1}^n A_i \rightarrow A_i$ for $i = 1, \dots, n$.
  \item The identities
        \[
         \pi_i \circ \iota_i = \id_{A_i} \hspace{3em} \text{and} \hspace{3em} \sum_{i=1}^n \iota_i \circ \pi_i = \id_{(\bigoplus_{i=1}^n A_i)}
        \]
        hold for all $i$.
 \end{enumerate}
\end{definition}

We say $\AC$ \emph{has direct sums} if it comes equipped with a function 
mapping a list of finitely many objects $A_1, \dots, A_n \in \AC$ to the data defining a direct sum of such a list.

\begin{remark}[Column convention matrix calculus]
    Given a degree $g \in \Z$ morphism
    \[
    \bigoplus_{i=1}^n A_i \stackrel{\alpha}{\longrightarrow} \bigoplus_{j=1}^m B_j,
    \]
    we have that
    \[
     \alpha_{ji} := \pi_j \circ \alpha \circ \iota_i: A_i \rightarrow B_j
    \]
    is a degree $g$ morphism.
    Furthermore, the identity
    \[
     \sum_{i,j} \iota_j \circ \alpha_{ji} \circ \pi_i = \alpha
    \]
    holds.
    Thus, such degree $g$ morphisms $\alpha$ are in one-to-one correspondence
    to matrices $(\alpha_{ji})_{ji}$ whose entries are degree $g$ morphisms.
\end{remark}

\begin{construction}
 Let $\AC$ be a dg category.
 We construct its \emph{completion by direct sums} $\AC^{\oplus}$
 as a dg category:
 \begin{enumerate}
  \item Objects are finite families of objects $A_1, \dots, A_n \in \AC$, $n \in \N_0$.
        We denote such families by $\bigoplus_{i=1}^n A_i$. Note that the empty family is explicitly allowed.
  \item Degree $g \in \Z$ morphisms from $\bigoplus_{i=1}^n A_i$ to $\bigoplus_{j=1}^m B_j$
        is the abelian group of $m\times n$ matrices $(\alpha_{ji})_{ji}$
        with $\alpha_{ji} \in ( A_i, B_j )^g$.
        We set
        \[
         d( (\alpha_{ji})_{ji} ) := (d\alpha_{ji})_{ji}.
        \]
  \item Composition is given by matrix multiplication:
        \[
         ( \beta_{kj} )_{kj} \circ (\alpha_{ji})_{ji} := ( \sum_{j}  \beta_{kj} \circ \alpha_{ji} )_{ki}.
        \]
  \item Identities are given by diagonal matrices:
        \[
         \id_{\bigoplus_{i=1}^n A_i} = \Diag\left( (\id_{A_i})_{i = 1, \dots, n} \right).
        \]
 \end{enumerate}
\end{construction}
\begin{proof}[Correctness of the construction]
 Composition is associative and the identities act as units since this is true for the matrix calculus.
 Last, we check the Leibniz rule:
 \begin{align*}
  d\left( ( \beta_{kj} )_{kj} \circ (\alpha_{ji})_{ji} \right) &= d\left( (\sum_j \beta_{kj} \circ \alpha_{ji})_{ki} \right) \\
  &=  \left(\sum_j d(\beta_{kj} \circ \alpha_{ji}) \right)_{ki}  \\
  &=  \left(\sum_j d(\beta_{kj}) \circ \alpha_{ji} + (-1)^{\abs{\beta}}\beta_{kj} \circ d(\alpha_{ji}) \right)_{ki} \\
  &=d( ( \beta_{kj} )_{kj} ) \circ (\alpha_{ji})_{ji} + (-1)^{\abs{\beta}}  ( \beta_{kj} )_{kj} \circ d((\alpha_{ji})_{ji}). \qedhere
 \end{align*}
\end{proof}

\begin{remark}
 It is easy to see that the completion by direct sums $\AC^{\oplus}$ actually has direct sums.
 Again, this is due to block matrix arithmetics.
\end{remark}

\begin{example}
    If $\AC$ is a $k$-category, then $\Hzero( \AC_{\dg}^{\oplus} )$
    is the completion by direct sums of $\AC$ as a $k$-category.
    In this sense, the concept of direct sums is lifted to the dg level.
\end{example}

\subsubsection{Completion by translations}\label{subsubsection:translations}

\begin{definition}
    Let $\AC$ be a dg category.
    Let $A \in \AC$ and $i \in \Z$.
    An \emph{$i$-th translation} of $\AC$ consists of the following data:
    \begin{enumerate}
     \item An object $A[i]$ in $\AC$.
     \item A closed isomorphism $A[i] \rightarrow A$ of degree $i$.
    \end{enumerate}
\end{definition}

We say $\AC$ \emph{has translations} if it comes equipped with a function 
mapping a pair $(A,i)$ to the data defining an $i$-th translation of $A$.

\begin{remark}
    In other words, the dg module $(-,A)$ shifted by $i$ is represented by $A[i]$.
\end{remark}

\begin{notation}\label{notation:translations}
    Let $\AC$ be a dg category with translations, $A, B$ be objects,
    $i,j \in \Z$.
    We use the notation
    \[
        A[i] \xlongrightarrow{ \transright{i} } A
    \]
    in order to denote the defining isomorphism of the $i$-th translation of $(A,i)$.
    Furthermore, we denote its inverse by
    \[
        A \xlongrightarrow{ \transleft{i} } A[i].
    \]
    Thanks to this notation, we may conveniently denote a morphism between
    two translations for $\alpha: A \rightarrow B$ in $\AC$ and $i,j \in \Z$ by
    \[
        A[i] \xrightarrow{ \trans{j}{\alpha}{i}  } B[j] := \transleft{j} \circ \alpha \circ \transright{i}.
    \]
\end{notation}

\begin{lemma}\label{lemma:formulas_translations}
    Let $\AC$ be a dg category with translations.
    Then, for morphisms $\alpha$, $\beta$ composable in $\AC$ and $i,j,l \in \Z$, the following equations hold:
    \begin{itemize}
        \setlength\itemsep{0.4em}
        \item $\abs{\trans{j}{\alpha}{i}} = \abs{\alpha} + i - j$,
        \item $d( {\trans{j}{\alpha}{i}} ) = (-1)^j {\trans{j}{d\alpha}{i}}$,
        \item ${\trans{l}{\beta}{j}} \circ {\trans{j}{\alpha}{i}} = {\trans{l}{(\beta \circ \alpha)}{i}}$,
        \item $\id_{A[i]} = {\trans{i}{(\id_A)}{i}}$.
    \end{itemize}
\end{lemma}
\begin{proof}
    The formulas for the degree, identity, and composition follow directly from the definitions.
    For the formula of the differential, we simply apply the Leibniz rule for $3$ terms ($d$ passing the morphism $\transleft{j}$ yields the sign $(-1)^j$)
    and use the fact that the differential of a closed isomorphism is zero.
\end{proof}
   
\begin{construction}
Let $\AC$ be a dg category.
We construct its \emph{completion by translations} $\AC^{[\bullet]}$
as a dg category:
\begin{enumerate}
    \item Objects are pairs $(A,i)$ consisting of an object $A \in \AC$ and an integer $i \in \Z$.
        We denote such a pair by $A[i]$.
    \item Morphisms from $A[i]$ to $B[j]$
        are given by morphisms $\alpha: A \rightarrow B$ in $\AC$.
        We denote such a morphism by ${\trans{j}{\alpha}{i}}: A[i] \rightarrow B[j]$.
    \item Degrees, identities, composition, and the differential are defined by the formulas
          stated in Lemma \ref{lemma:formulas_translations}.
\end{enumerate}
\end{construction}
\begin{proof}[Correctness of the construction]
Composition is associative and the identities act as units since this is true in the underlying dg category $\AC$.
We need to check the Leibniz rule:
\begin{align*}
    d\left( {\trans{l}{\beta}{j}} \circ {\trans{j}{\alpha}{i}} \right) &= d\left( {\trans{l}{\beta \circ \alpha}{i}} \right) \\
    &= (-1)^l  \trans{l}{d(\beta \circ \alpha)}{i}  \\
    &=  (-1)^l  \trans{l}{ \left(d(\beta) \circ \alpha + (-1)^{\abs{\beta}} \beta \circ d\alpha\right)  }{i} \\
    &=  (-1)^l  \trans{l}{ \left(d(\beta) \circ \alpha\right) }{i} + (-1)^{l + \abs{\beta}} \trans{l}{ \left(\beta \circ d\alpha\right)  }{i} \\
    &=  d(\trans{l}{ \beta }{j}) \circ \trans{j}{\alpha }{i} + (-1)^{\abs{ \trans{l}{\beta}{j} }} \trans{l}{ \beta }{j} \circ d(\trans{j}{\alpha}{i}). \qedhere
\end{align*}
\end{proof}
   
\begin{remark}
The completion by translations $\AC^{[\bullet]}$ actually has translations:
let $A[i] \in \AC^{[\bullet]}$ and $j \in \Z$, then we can set $A[i][j] := A[i+j]$
and $\trans{i}{(\id_A)}{i+j}: A[ i + j ] \rightarrow A[ i ]$ is a closed isomorphism of degree $j$.
Moreover, the notation by angular brackets introduced in Notation \ref{notation:translations}
applied to $\AC^{[\bullet]}$ as a dg category with translations
yields the same result as its defining symbolic notation with angular brackets, thus, there is no need for a formal distinction.
\end{remark}

\begin{example}
    If $\AC$ is a $k$-category, then $\Hzero( (\AC_{\dg}^{[\bullet]})^{\oplus} )$
    is equivalent to the category of complexes with zero differentials
    whose objects are formal direct sums of objects in $\AC$.
\end{example}

\subsubsection{Completion by twisted complexes}\label{subsubsection:twisted_complexes}
\begin{definition}
    Let $\AC$ be a dg category.
    A \emph{twisted complex} of an object $A \in \AC$ and
    an endomorphism $q: A \rightarrow A$ with $\abs{ q } = 1$
    satisfying the Maurer Cartan equation
    \[
        d(q) + q ^2 = 0
    \]
    consists of the following data:
    \begin{enumerate}
        \item An object $\twistcobj{A}{q}$ in $\AC$.
        \item An isomorphism $\twistcobj{A}{q} \xrightarrow{\iota} A$ of degree $0$ (not necessarily closed).
        \item The following equation holds: \[d(\iota) + q \circ \iota = 0.\]
    \end{enumerate}
\end{definition}

We say $\AC$ \emph{has twisted complexes} if it comes equipped with a function 
mapping a pair $(A,q)$ to the data defining a twisted complex of such a pair.

\begin{remark}
    In other words, $\twistcobj{A}{q}$ represents the dg module 
    $(-,A)$ with twisted differential given by $\alpha \mapsto q \circ \alpha + d(\alpha)$
    for all $B \in \AC, \alpha \in (B,A)$.
\end{remark}

\begin{example}
    Let $V$ be a graded space considered as an object in the dg category of complexes $\Chbdg( k )$. Let $q: V \rightarrow V$
    be a graded map of degree $1$ such that $q^2 = 0$.
    The twisted complex of the pair $(V,q)$ in $\Chbdg( k )$ is given by the complex
    \[
        \dots \rightarrow V^i \xrightarrow{q^i} V^{i+1} \rightarrow \dots
    \]
    together with the not necessarily closed map
    \begin{center}
        \begin{tikzpicture}[label/.style={postaction={
        decorate,
        decoration={markings, mark=at position .5 with \node #1;}},
        mylabel/.style={thick, draw=none, align=center, minimum width=0.5cm, minimum height=0.5cm,fill=white}}]
        \coordinate (r) at (3,0);
        \coordinate (d) at (0,-2);
        \node (L) {$\dots$};
        \node (A) at ($(L) + (r)$) {$V^i$};
        \node (B) at ($(A)+(r)$) {$V^{i+1}$};
        \node (C) at ($(B) + (r)$){$\dots$,};
        
        \node (L2) at ($(L) + (d)$) {$\dots$};
        \node (A2) at ($(A) + (d)$) {$V^i$};
        \node (B2) at ($(B)+(d)$) {$V^{i+1}$};
        \node (C2) at ($(C) + (d)$){$\dots$};
        \draw[->,thick] (L) -- (A);
        \draw[->,thick,label={[above]{$q^i$}}] (A) -- (B);
        \draw[->,thick,label={[above]{$0$}}] (A2) -- (B2);
        \draw[->,thick] (B) -- (C);
        \draw[->,thick] (B2) -- (C2);
        \draw[->,thick] (L2) -- (A2);
        
        \draw[->,thick,label={[right]{$\id$}}] (A) -- (A2);
        \draw[->,thick,label={[right]{$\id$}}] (B) -- (B2);
        \end{tikzpicture}
    \end{center}
    In other words, the formal concept of twisted complexes allows us to pass from graded spaces to complexes.
\end{example}

\begin{notation}\label{notation:complexes}
    Let $\AC$ be a dg category with twisted complexes, $A, B$ be objects,
    $A \xrightarrow{p} A$, 
    $B \xrightarrow{q} B$
    be endomorphisms of degree $1$ satisfying the Maurer Cartan equation.
    We use the notation
    \[
        \twistcobj{A}{p} \xrightarrow{ \twistcright{p}} A
    \]
    in order to denote the defining isomorphism of the twisted complex of $(A,p)$.
    Furthermore, we denote its inverse by
    \[
        A \xrightarrow{ \twistcleft{p} } \twistcobj{A}{p}.
    \]
    Thanks to this notation, we may conveniently denote a morphism between
    two twisted complexes by
    \[
        \twistcobj{A}{p} \xrightarrow{ \twistc{q}{\alpha}{p}  } \twistcobj{B}{q} := \twistcleft{q} \circ \alpha \circ \twistcright{p}
    \]
    where $\alpha: A \rightarrow B$ is a morphism in $\AC$.
\end{notation}

\begin{lemma}\label{lemma:formulas_twisted_complexes}
    Let $\AC$ be a dg category with twisted complexes.
    Then, the following equations hold:
    \begin{enumerate}
        \setlength\itemsep{0.4em}
        \item $\twistc{p}{(\id_A)}{p} = \id_{(\twistcobj{A}{p})}$
        \item $\abs{ \twistc{q}{\alpha}{p} } = \abs{ \alpha }$
        \item $\twistc{r}{\beta}{q} \circ \twistc{q}{\alpha}{p} = \twistc{r}{(\beta \circ \alpha)}{p}$
        \item $d( \twistc{q}{\alpha}{p} ) = \twistc{q}{(q \circ \alpha)}{p} + \twistc{q}{d(\alpha)}{p}  + (-1)^{\abs{ \alpha } + 1}\twistc{q}{(\alpha \circ p)}{p} $
    \end{enumerate}
    for all composable morphisms $\alpha, \beta$ and endomorphisms $p,q,r$ of degree $1$ satisfying the Maurer Cartan equation.
\end{lemma}
\begin{proof}
    Again, the formulas for the degree, identity, and composition follow directly from the definitions.
    In order to prove the formula of the differential,
    first we note that by the definition of a twisted complex, we have
    \[d( \twistcright{p} ) = -p \circ \twistcright{p}.\]
    As a consequence of the Leibniz rule and the fact that identities are closed, the differential of the inverse morphism is given by
    \[d( \twistcleft{q} ) = \twistcleft{q} \circ q.\]
    From these two equations and the Leibniz rule on $3$ terms, the claim follows.
\end{proof}

\begin{construction}
    Let $\AC$ be a dg category.
    We construct its \emph{completion by twisted complexes}
    $\Twist( \AC )$ as a dg category:
    \begin{enumerate}
    \item Objects in $\Twist( \AC )$
    are given by pairs $(A, p)$
    of objects $A \in \AC$ and endomorphisms $A \xrightarrow{p} A$
    of degree $1$ satisfying the Maurer Cartan equation.
    We denote such a pair by $\twistcobj{A}{p}$.
    \item Morphisms from $\twistcobj{A}{p}$ to $\twistcobj{B}{q}$ are simply given
    by morphisms $\alpha: A \rightarrow B$ in $\AC$. Whenever we consider a morphism $\alpha$ in $\AC$
    as such a morphism in $\Twist( \AC )$, we denote it by $\twistc{q}{\alpha}{p}$.
    \item Degrees, identities, composition, and the differential are defined by the formulas stated in Lemma \ref{lemma:formulas_twisted_complexes}.
    \end{enumerate}
\end{construction}
\begin{proof}[Correctness of the construction]
    Composition is associative and the identities act as units since this is true in the underlying dg category $\AC$.
We need to check the Leibniz rule:
\begin{align*}
    &{\phantom{=}} d\left( {\twistc{r}{\beta}{q}} \circ {\twistc{q}{\alpha}{p}} \right) \\[1em]
    &= d\left( {\twistc{r}{(\beta \circ \alpha)}{p}} \right) \\[1em]
    &= {\twistc{r}{(r \circ \beta \circ \alpha)}{p}} + {\twistc{r}{d(\beta \circ \alpha)}{p}} + (-1)^{\abs{\beta \circ \alpha} + 1} {\twistc{r}{(\beta \circ \alpha \circ p)}{p}} \\[1em]
    &= {\twistc{r}{(r \circ \beta \circ \alpha)}{p}} + {\twistc{r}{(d(\beta) \circ \alpha)}{p}} + {\twistc{r}{((-1)^{\abs{\beta}}\beta \circ d(\alpha))}{p}} + (-1)^{\abs{\beta \circ \alpha} + 1} {\twistc{r}{(\beta \circ \alpha \circ p)}{p}} \\[1em]
    &= \left({\twistc{r}{(r \circ \beta)}{q}} + {\twistc{r}{(d\beta)}{q}}\right) \circ \twistc{q}{\alpha}{p}\\
    &\hspace{2em}+ (-1)^{\abs{\beta}}\twistc{r}{\beta}{q} \circ \left({\twistc{q}{( d\alpha)}{p}} + (-1)^{\abs{\alpha} + 1} {\twistc{q}{(\alpha \circ p)}{p}}\right) \\[1em]
    &= \left({\twistc{r}{(r \circ \beta)}{q}} + {\twistc{r}{(d\beta)}{q}} + (-1)^{\abs{\beta} + 1}\twistc{r}{(\beta \circ q)}{q}\right) \circ \twistc{q}{\alpha}{p}\\
    &\hspace{2em}+ (-1)^{|{\twistc{r}{\beta}{q}}|}\twistc{r}{\beta}{q} \circ \left(\twistc{q}{(q \circ \alpha)}{p} +  {\twistc{q}{( d\alpha)}{p}} + (-1)^{\abs{\alpha} + 1} {\twistc{q}{(\alpha \circ p)}{p}}\right) \\[1em]
    &= d( \twistc{r}{\beta}{q}) \circ \twistc{q}{\alpha}{p} + (-1)^{| \twistc{r}{\beta}{q}  |}\twistc{r}{\beta}{q} \circ d( \twistc{q}{\alpha}{p}) \qedhere
\end{align*}
\end{proof}

\begin{lemma}
    Let $\AC$ be a dg category. Then $\Twist(\AC)$ has twisted complexes.
\end{lemma}
\begin{proof}
    Suppose we are given an endomorphism in $\Twist( \AC )$
    \[
        \twistcobj{A}{p} \xrightarrow{ \twistc{p}{\alpha}{p} } \twistcobj{A}{p}
    \]
    of degree $1$ satisfying the Maurer Cartan equation.
    We claim that this gives rise to the twisted complex
    with object $\twistcobj{A}{p + \alpha}$
    and isomorphism
    \[ \twistcobj{A}{p + \alpha} \xrightarrow{ \twistc{p}{(\id_A)}{p + \alpha} } \twistcobj{A}{p}.\]
    First, we need to see that $p + \alpha$ satisfies the Maurer Cartan equation.
    Since $\twistc{p}{\alpha}{p}$ satisfies the Maurer Cartan equation, we have
    \begin{align*}
        d(\twistc{p}{\alpha}{p}) + (\twistc{p}{\alpha}{p})^2 =
        \twistc{p}{(p \circ \alpha + d\alpha + (\alpha \circ p) + \alpha^2)}{p} = 0
    \end{align*}
    which implies $((p \circ \alpha) + d\alpha + (\alpha \circ p) + \alpha^2) = 0$. Moreover, since $dp + p^2 = 0$, we get
    \begin{align*}
        d( p + \alpha ) + (p + \alpha)^2 = dp + d\alpha + p^2 + p \circ \alpha + \alpha \circ p + \alpha^2 = 0.
    \end{align*}
    Second, we need to compute the differential of the desired isomorphism in order to see that is satisfies the defining property of a twisted complex:
    \begin{align*}
        d( \twistc{p}{(\id_A)}{p + \alpha} ) &= \twistc{p}{p}{p + \alpha} - \twistc{p}{(p+\alpha)}{p+\alpha}\\
        &= (-\twistc{p}{\alpha}{p+\alpha}) = (-\twistc{p}{\alpha}{p+\alpha}) \circ  \twistc{p}{(\id_A)}{p + \alpha} \qedhere
    \end{align*}
\end{proof}
    
\begin{remark}
    The notation by curly brackets introduced in Notation \ref{notation:complexes}
    applied to $\Twist(\AC)$ as a dg category with twisted complexes
    yields the same result as its defining symbolic notation with curly brackets, thus, there is no need for a formal distinction.
\end{remark}
    
\begin{example}\label{example:twist_for_k_cat}
    If $\AC$ is a $k$-category, then $\Twist((\AC_{\dg}^{[\bullet]})^{\oplus}) \simeq \Chbdg( \AC )$.
\end{example}

\subsubsection{The pretriangulated hull}\label{subsubsection:pretr}

Let $\AC$ be a dg category.
Its \emph{pretriangulated hull} $\pretr(\AC)$ can be constructed explicitly as the full dg subcategory
\[
    \pretr(\AC) \subseteq \Twist((\AC^{[\bullet]})^{\oplus})
\]
generated by those twisted complexes
\[
    \twistcobj{ \big(\bigoplus_{i = 1}^n A_i[ t_i ]\big) }{ (A_i[ t_i ] \xrightarrow{q_{ji}} A_j[ t_j ])_{ji} }
\]
with $n \in \N_0$, $A_i \in \AC$, $t_i \in \Z$ for $i = 1, \dots, n$,
such that $q = (q_{ji})_{ji}$ is a lower triangular matrix, i.e., $q_{ji} = 0$ if $j \geq i$ (cf. \cite[Definition 2.2]{Bod15}).
The objects in $\pretr(\AC)$ are called \emph{one-sided twisted complexes}.

More generally, we could also consider the full dg subcategory $\BC$ of $\Twist((\AC^{[\bullet]})^{\oplus})$
generated by those twisted complexes for which there exists a permutation of the summands such that
the resulting $q$ is a lower triangular matrix.
Since this permutation clearly defines a dg isomorphism to the original object, $\BC$ is equivalent to $\pretr(\AC)$.
This observation and Example \ref{example:twist_for_k_cat} yield the following lemma.

\begin{lemma}
    Let $\AC$ be a $k$-category.
    Then the inclusion $\pretr(\AC) \subseteq \Twist((\AC^{[\bullet]})^{\oplus})$
    yields the chain of equivalences
    \[
        \pretr(\AC_{\dg} ) \simeq \Twist((\AC_{\dg}^{[\bullet]})^{\oplus}) \simeq \Chbdg( \AC ).
    \]
\end{lemma}

\subsection{Formal constructions of \Ainf-functors}\label{subsection:formal_a_inf_functors}

We follow the $b$-convention in our presentation of the theory of $A_{\infty}$-categories,
see als \cite{Ainf03} for a very detailed account.
Please recall the notational conventions concerning graded spaces given at the end of Section \ref{section:introduction}.

\subsubsection{Preliminaries}

We recall the definitions of \Ainf-categories and \Ainf-functors.

\begin{definition}\label{definition:a_inf_cat}
    An $A_{\infty}$-category $\AC$ consists of the following data:
    \begin{enumerate}
        \item A collection of objects $\Obj_{\AC}$.
        \item For all objects $A_1, A_2$, a graded space $(A_1, A_2)$, called \emph{morphisms from $A_1$ to $A_2$}.
        \item For every $n \geq 1$ and objects $A_1, \dots, A_{n+1}$, a graded map
        \[
            b_n: S(A_n, A_{n+1}) \otimes \dots \otimes S(A_1, A_2) \longrightarrow S(A_1, A_{n+1})
        \]
        of degree $1$.
        \item For every $n \geq 1$, the equation
        \begin{equation}\label{equation:n_gen_b_gen_zero}
            \sum_{ \substack{i + j + l = n \\ (j \geq 1) \\ (i,l \geq 0) } } b_{i + 1 + l} \circ ( 1^{\otimes{i}} \otimes b_j \otimes 1^{\otimes{l}} ) = 0
        \end{equation}
        holds, where $1^{\otimes{i}}$ denotes the identity of the $i$ factors to the left, i.e., 
        the identity of $S(A_n, A_{n+1}) \otimes \dots \otimes S(A_{n-i+1}, A_{n-i+2})$,
        and $1^{\otimes{l}}$ denotes the identity of the $l$ factors to the right, i.e., 
        the identity of $S(A_{l}, A_{l+1}) \otimes \dots \otimes S(A_{1}, A_{2})$.
    \end{enumerate}
    Usually, the definition of an $A_{\infty}$-category ends here.
    However, since we need strict identities, we will make them part of the definition.
    \begin{enumerate}
        \setcounter{enumi}{4}
        \item For all objects $A$, a morphism $\id_A \in (A,A)$ of degree $0$, called \emph{strict identity}.
        \item For all objects ${A_1},{A_2}$ and $\alpha \in ({A_1},{A_2})$, we have
        \begin{itemize}
            \setlength\itemsep{0.4em}
            \item $b_2( \degdown{\alpha}, \degdown{\id_{A_1}} ) = (-1)^{\abs{\alpha}}\degdown{\alpha}$,
            \item $b_2( \degdown{\id_{A_2}}, \degdown{\alpha} ) = \degdown{\alpha}$,
            \item $b_i( \dots, \degdown{\id_{A_1}}, \dots ) = 0$ for $i \neq 2$, i.e., it vanishes whenever a strict identity occurs as an argument.
        \end{itemize}
    \end{enumerate}
\end{definition}

\begin{remark}\label{remark:special_cases}
    We look at Equation \eqref{equation:n_gen_b_gen_zero} in small and special cases:
    \begin{center}
        \begin{tabular}{c|c|c}
            $\mathbf{n}$ & special case & formula \\
            \hline
            $1$ & - &  $b_1 \circ b_1 = 0$\\
            \hline
            $2$ & - &  $b_2 \circ ( 1 \otimes b_1) + b_2 \circ ( b_1 \otimes 1) + b_1 \circ b_2 = 0$ \\
            \hline
            $3$ & $b_3 = 0$ & $b_2 \circ ( 1 \otimes b_2) + b_2 \circ ( b_2 \otimes 1) = 0$
        \end{tabular}
    \end{center}
\end{remark}

\begin{construction}[From dg categories to $A_{\infty}$-categories]\label{construction:dg_to_Ainf}
    Let $\AC$ be a dg category.
    Then we construct an $A_{\infty}$-category $\AC_{\infty}$ as follows:
    \begin{itemize}
        \item $\Obj_{\AC} = \Obj_{\AC_{\infty}}$ and the same for all morphism spaces and identities.
        \item For objects $A_1, A_2$ and morphism $\alpha \in (A_1, A_2)$, we set
        \[ b_1(\degdown{\alpha}) := -\degdown{d(\alpha)}. \]
        \item For objects $A_1, A_2, A_3$ and morphisms
        $\alpha \in (A_2, A_3), \beta \in (A_1, A_2)$, we set
        \[b_2( \degdown{\alpha} \otimes \degdown{\beta} ) := (-1)^{\abs{\alpha}} \degdown{( \alpha \circ \beta )}.\]
        \item For $n \geq 3$, we set $b_n := 0$.
    \end{itemize}
    
    Conversely, every $A_{\infty}$-category $\BC$ with $b_i = 0$ for $i > 2$ gives rise to a dg category
    with the same objects and graded morphism spaces, and
    \begin{itemize}
        \setlength\itemsep{0.4em}
        \item $d\beta := (-1)\degup{b_1(\degdown{\beta})}$ for $\beta \in (B_1, B_{2})$,
        \item $\alpha \circ \beta := (-1)^{\abs{\alpha}}\degup{b_2( \degdown{\alpha} \otimes \degdown{\beta} )}$ for $\beta \in (B_1, B_{2})$, $\alpha \in (B_2, B_3)$,
    \end{itemize}
    where $B_1, B_2, B_3 \in \BC$.
\end{construction}
\begin{proof}[Correctness of the construction]
    The special cases enlisted in Remark \ref{remark:special_cases} translate exactly to the defining equations of a dg category.
\end{proof}

\begin{definition}\label{definition:a_inf_func}
    Let $\AC$, $\BC$ be $A_{\infty}$-categories.
    An $A_{\infty}$-functor $f: \AC \rightarrow \BC$ consists of the following data:
    \begin{enumerate}
        \item A map $f: \Obj_{\AC} \rightarrow \Obj_{\BC}$.
        \item For every $n \geq 1$ and objects $A_1, \dots, A_{n+1}$ in $\AC$, a graded map
        \[
            f_n: S(A_n, A_{n+1}) \otimes \dots \otimes S(A_1, A_2) \longrightarrow S(fA_1, fA_{n+1})
        \]
        of degree $0$.
        \item For every $n \geq 1$, the equation \[
            \sum_{\substack{i + j + l = n \\ (j \geq 1) \\ (i,l \geq 0)}}f_{i + 1 + l} \circ ( 1^{\otimes i} \otimes b_j \otimes 1^{\otimes l} ) 
            = 
            \sum_{
                \substack{i_1 + \dots + i_s = n \\ s \geq 1 \\ (i_1, \dots, i_s \geq 1)}
                } 
                b_s \circ (f_{i_1} \otimes \dots \otimes f_{i_s})
        \]
        holds ($f_{\bullet}$ commutes with $b_{\bullet}$, and the equations are linear w.r.t.\ the $b$-maps).
        \item For every object $A$ in $\AC$, we have \[
            f_1( \degdown{\id_A} ) = \degdown{\id_{fA}}
        \]
        and \[
            f_n( \dots, \degdown{\id_A}, \dots ) = 0
        \]
        for $n \geq 2$ (strict units are strictly respected).
        
    \end{enumerate}
\end{definition}

\begin{remark}\label{remark:main_equation}
    Let $\AC$, $\BC$ be dg categories.
An $A_{\infty}$-functor $f: \AC_{\infty} \rightarrow \BC_{\infty}$ satisfies
\begin{equation}\label{equation:main}
    \begin{split}
        &\phantom{{=}} \overbrace{\sum_{{i = 0}}^{n-1}f_{n} \circ \big( 1^{\otimes i} \otimes b_1 \otimes 1^{\otimes {n-(i+1)}} \big)}^{\text{(A)-term}}
        +
        \overbrace{\sum_{{i = 0}}^{n-2}f_{n-1} \circ \big( 1^{\otimes i} \otimes b_2 \otimes 1^{\otimes {n-(i+2)}} \big)}^{\text{(B)-term}}
         \\
        &{=} \underbrace{\sum_{i=1}^{n-1} b_2 \circ \big( f_{i} \otimes f_{n-i} \big)}_{\text{(C)-term}}
        +
        \underbrace{b_1 \circ f_n}_{\text{(D)-term}}
    \end{split}
\end{equation}
which means that for every $n \geq 1$,
the maps $f_1, \dots, f_{n-1}$ are compatible with $b_2$
up to a term specified by $f_n$ and $b_1$.
The special case $n=1$ yields sums over empty terms and thus the special formula
\[
    f_1 \circ b_1 = b_1 \circ f_1.
\]
In order to make all terms well-defined, we may set $f_0 := 0$.
\end{remark}

\begin{definition}
    Any \Ainf-functor $f: \AC \rightarrow \BC$ between dg categories induces a functor 
    $\Hzero(f): \Hzero( \AC ) \rightarrow \Hzero( \BC)$
    on their corresponding homotopy categories by setting
    \[
        \Hzero(f)( A ) \coloneqq f(A)
    \]
    on objects $A \in \AC$ and
    \[
        \Hzero(f)( \resclass{\alpha} ) \coloneqq \resclass{\degup{f_1(\degdown{\alpha})}}
    \]
    on morphisms $\alpha \in \AC$.
\end{definition}

\subsubsection{Lift to completion by direct sums}\label{subsubsection:lift_direct_sums}

\begin{notation}
    Applying Fukaya's sign convention to the following term yields
\[
    (f \otimes g \otimes h)( \alpha \otimes \beta \otimes \gamma) = (-1)^{\abs{ \alpha }\abs{ g } + \abs{ \alpha }\abs{ h } + \abs{ h }\abs{ \beta } } f( \alpha ) \otimes g( \beta ) \otimes h( \gamma ).
\]
A typical occurrence of this situation is the term
\[
    (1^{\otimes i} \otimes b_j \otimes 1^{\otimes l} )( \degdown{\alpha_n} \otimes \dots \otimes \degdown{\alpha_1})
\]
yielding the sign
\[
    \sigma(i) := (-1)^{ \sum_{j= 0}^{i-1} \abs{ \degdown{\alpha_{n-j}}}} = (-1)^{ \sum_{j= 0}^{i-1} \abs{ \alpha_{n-j}} - i}.
\]
We denote this sign simply by $\sigma(i)$ since in all contexts
in which we will use this sign, the prescribed sequences of morphisms
will be clear from context.
\end{notation}

\begin{notation}
    If $\BC$ is a dg category with direct sums,
    and $( \beta_{ji} )_{ji}$ a morphism in $\BC$ given by a matrix,
    we simplify notation w.r.t.\ degree shifts by $( \degdown{\beta_{ji}} )_{ji} \coloneqq \degdown{( \beta_{ji} )_{ji}}$.
\end{notation}

\begin{construction}\label{construction:lift_direct_sums}
    Let $\AC$ be a dg category and $\BC$ be a dg category with direct sums.
    Let $f: \AC_{\infty} \longrightarrow \BC_{\infty}$ be an $A_{\infty}$-functor.
We are going to construct an $A_{\infty}$-functor
\[
    F: (\AC^{\oplus})_{\infty} \longrightarrow \BC_{\infty}
\]
extending $f$ w.r.t.\ the natural inclusion $\AC_{\infty} \hookrightarrow (\AC^{\oplus})_{\infty}$.
We set 
\[
    F( \oplus_{i=1}^s A_i ) := \oplus_{i=1}^s f(A_i)
\]
on objects.
Moreover, for $n \geq 1$, suppose given a chain of morphisms 
\begin{center}
    \begin{tikzpicture}[label/.style={postaction={
    decorate,
    decoration={markings, mark=at position .5 with \node #1;}},
    mylabel/.style={thick, draw=none, align=center, minimum width=0.5cm, minimum height=0.5cm,fill=white}}]
    \coordinate (r) at (4,0);
    \node (A) {$\oplus_{i_{n+1}} A_{i_{n+1}}$};
    \node (B) at ($(A)+(r)$) {$\oplus_{i_{n}} A_{i_{n}}$};
    \node (C) at ($(B) + 0.5*(r)$){$\dots$};
    \node (D) at ($(C)+0.4*(r)$) {$\oplus_{i_{2}} A_{i_{2}}$};
    \node (E) at ($(D)+(r)$) {$\oplus_{i_{1}} A_{i_{1}}$};
    
    \draw[->,thick,label={[above]{$\big( \alpha^{n}_{i_{n+1} i_{n}}\big)_{i_{n+1} i_{n}}$}}] (B) -- (A);
    \draw[->,thick] (C) -- (B);
    
    \draw[->,thick] (D) -- (C);
    \draw[->,thick,label={[above]{$\big( \alpha^{1}_{i_2 i_{1}}\big)_{i_2 i_{1}}$}}] (E) -- (D);
    \end{tikzpicture}
\end{center}
in $\AC^{\oplus}$. For all $n \geq 1$, we set 
\begin{align*}
    F_n\left(  \degdown{(\alpha^{n}_{i_{n+1} i_{n}})_{i_{n+1} i_{n}}}, \dots, \degdown{( \alpha^{1}_{i_2 i_{1}})_{i_2 i_{1}}} \right) :=
    \left(
        \sum_{(i_2, \dots, i_n)}
        f_n\left( \degdown{\alpha^{n}_{i_{n+1} i_{n}}}, \dots, \degdown{\alpha^{1}_{i_2 i_{1}}} \right)
    \right)_{i_{n+1}i_1}
\end{align*}
on such chains of morphisms.
Note that the sum is taken over all tuples $(i_2, \dots, i_n)$
that constitute
of indices occurring in the intermediate objects. In the special case $n = 1$,
there only is one such tuple, namely the empty tuple $()$,
and the definition degenerates to
\begin{align*}
    F_1\left( \degdown{(\alpha^{1}_{i_{2} i_{1}})_{i_{2} i_{1}}} \right) =
    \left(f_1 \degdown{\alpha^{1}_{i_{2} i_{1}}}\right)_{i_{2} i_{1}}.
\end{align*}

\end{construction}
\begin{proof}[Correctness of construction]
    Clearly, $F$ extends $f$.
    Furthermore, it is compatible with strict identities.
    For $n \geq 1$, we need to show that Equation \eqref{equation:main} holds.
    We evaluate the (A)-(D) terms of Equation \eqref{equation:main} at $\big(  \degdown{(\alpha^{n}_{i_{n+1} i_{n}})_{i_{n+1} i_{n}}} , \dots , \degdown{(\alpha^{1}_{i_{2} i_{1}})_{i_{2} i_{1}}} \big)$.
    \begingroup
    \\ \\
    \allowdisplaybreaks
    (A)-term:
    {
    \begin{align*}
    &\phantom{=}
        \sum_{{i = 0}}^{n-1}F_{n} \circ \big( 1^{\otimes i} \otimes b_1 \otimes 1^{\otimes {n-(i+1)}} \big)
        \big(  \degdown{(\alpha^{n}_{i_{n+1} i_{n}})_{i_{n+1} i_{n}}} , \dots , \degdown{(\alpha^{1}_{i_{2} i_{1}})_{i_{2} i_{1}}} \big)
    \\[1em]
    &{=}
    \sum_{{i = 0}}^{n-1}F_{n}
    \big(  \degdown{(\alpha^{n}_{i_{n+1} i_{n}})_{i_{n+1} i_{n}}}, \dots b_1\degdown{( \alpha^{n-i}_{i_{n+1-i} i_{n-i}})_{i_{n+1-i} i_{n-i}}}, \dots \degdown{(\alpha^{1}_{i_{2} i_{1}})_{i_{2} i_{1}}} \big)(-1)^{\sigma(i)} \\[1em]
    &{=}
    \sum_{{i = 0}}^{n-1}
    \big(
    \sum_{(i_2, \dots, i_n)}
    f_{n}
    \big(  \degdown{\alpha^{n}_{i_{n+1} i_{n}}}, \dots b_1\degdown{\alpha^{n-i}_{i_{n+1-i} i_{n-i}}}, \dots \degdown{\alpha^{1}_{i_{2} i_{1}}} \big)
    \big)_{i_{n+1}i_1}
    (-1)^{\sigma(i)} \\[1em]
    &{=}
    \sum_{(i_2, \dots, i_n)}
    \big(
    \sum_{{i = 0}}^{n-1}
    f_{n}
    \big( \degdown{\alpha^{n}_{i_{n+1} i_{n}}}, \dots b_1\degdown{\alpha^{n-i}_{i_{n+1-i} i_{n-i}}}, \dots \degdown{\alpha^{1}_{i_{2} i_{1}}} \big)(-1)^{\sigma(i)}
    \big)_{i_{n+1}i_1}\\[1em]
    &{=}
    \sum_{(i_2, \dots, i_n)}
    \big(
    \sum_{{i = 0}}^{n-1}
    f_{n} \circ \big( 1^{\otimes i} \otimes b_1 \otimes 1^{\otimes {n-(i+1)}} \big)
    \big( \degdown{\alpha^{n}_{i_{n+1} i_{n}}}, \dots, \degdown{\alpha^{1}_{i_2 i_{1}}} \big)
    \big)_{i_{n+1}i_1}
    \end{align*}
    }
    \\ \\
    (B)-term:
    {
        \tiny
    \begin{align*}
    &\phantom{=}
        \sum_{{i = 0}}^{n-2}F_{n-1} \circ \big( 1^{\otimes i} \otimes b_2 \otimes 1^{\otimes {n-(i+2)}} \big)
        \big(  \degdown{(\alpha^{n}_{i_{n+1} i_{n}})_{i_{n+1} i_{n}}} , \dots , \degdown{(\alpha^{1}_{i_{2} i_{1}})_{i_{2} i_{1}}} \big)
    \\[1em]
    &{=}
    \sum_{{i = 0}}^{n-2}F_{n-1}
    \big(  \degdown{(\alpha^{n}_{i_{n+1} i_{n}})_{i_{n+1} i_{n}}}, 
    \dots b_2\big(\degdown{( \alpha^{n-i}_{i_{n+1-i} i_{n-i}})_{i_{n+1-i} i_{n-i}}}, \degdown{( \alpha^{n-i-1}_{i_{n-i} i_{n-i-1}})_{i_{n-i} i_{n-i-1}}} \big),
    \dots \degdown{(\alpha^{1}_{i_{2} i_{1}})_{i_{2} i_{1}}} \big)(-1)^{\sigma(i)} \\[1em]
    &{=}
    \sum_{{i = 0}}^{n-2}F_{n-1}
    \big(  \degdown{(\alpha^{n}_{i_{n+1} i_{n}})_{i_{n+1} i_{n}}}, 
    \dots {\big(\sum_{i_{n-i}}{b_2 (\degdown{ \alpha^{n-i}_{i_{n+1-i} i_{n-i}}}, \degdown{ \alpha^{n-i-1}_{i_{n-i} i_{n-i-1}}})} \big)_{i_{n-i+1}, i_{n-i-1}}},
    \dots \degdown{(\alpha^{1}_{i_{2} i_{1}})_{i_{2} i_{1}}} \big)(-1)^{\sigma(i)} \\[1em]
    &{=}
    \sum_{{i = 0}}^{n-2}
    \big(
    \sum_{(i_2, \dots,\widehat{i_{n-i}}, \dots, i_n)}
    f_{n-1}
    \big(  \degdown{\alpha^{n}_{i_{n+1} i_{n}}},
    \dots \sum_{i_{n-i}}b_2\big(\degdown{\alpha^{n-i}_{i_{n+1-i} i_{n-i}}}, \degdown{\alpha^{n-i-1}_{i_{n-i} i_{n-i-1}}} \big),
    \dots \degdown{\alpha^{1}_{i_{2} i_{1}}} \big)
    \big)_{i_{n+1}i_1}
    (-1)^{\sigma(i)} \\[1em]
    &{=}
    \sum_{{i = 0}}^{n-2}
    \big(
    \sum_{(i_2, \dots, i_n)}
    f_{n-1}
    \big(  \degdown{\alpha^{n}_{i_{n+1} i_{n}}},
    \dots b_2\big(\degdown{\alpha^{n-i}_{i_{n+1-i} i_{n-i}}}, \degdown{\alpha^{n-i-1}_{i_{n-i} i_{n-i-1}}} \big),
    \dots \degdown{\alpha^{1}_{i_{2} i_{1}}} \big)
    \big)_{i_{n+1}i_1}
    (-1)^{\sigma(i)} \\[1em]
    &{=}
    \sum_{(i_2, \dots, i_n)}
    \big(
    \sum_{{i = 0}}^{n-2}
    f_{n-1}
    \big(  \degdown{\alpha^{n}_{i_{n+1} i_{n}}},
    \dots b_2\big(\degdown{\alpha^{n-i}_{i_{n+1-i} i_{n-i}}}, \degdown{\alpha^{n-i-1}_{i_{n-i} i_{n-i-1}}} \big),
    \dots \degdown{\alpha^{1}_{i_{2} i_{1}}} \big)
    (-1)^{\sigma(i)}\big)_{i_{n+1}i_1}
     \\[1em]
     &{=}
     \sum_{(i_2, \dots, i_n)}
     \big(
     \sum_{{i = 0}}^{n-2}f_{n-1} \circ \big( 1^{\otimes i} \otimes b_2 \otimes 1^{\otimes {n-(i+1)}} \big)
    \big( \degdown{\alpha^{n}_{i_{n+1} i_{n}}} , \dots , \degdown{\alpha^{1}_{i_{2} i_{1}}} \big)
    \big)_{i_{n+1}i_1}
    \end{align*}
    }
    \\ \\
    (C)-term:
    {
    \begin{align*}
        &\phantom{=}\sum_{i=0}^{n-1}b_2 \circ (F_i \otimes F_{n-i})\big(  \degdown{(\alpha^{n}_{i_{n+1} i_{n}})_{i_{n+1} i_{n}}} , \dots , \degdown{(\alpha^{1}_{i_{2} i_{1}})_{i_{2} i_{1}}} \big)\\[1em]
        &{=}
        \sum_{i=0}^{n-1}b_2 \big(F_i( \degdown{(\alpha^{n}_{i_{n+1} i_{n}})_{i_{n+1} i_{n}}}, \dots ), F_{n-i}( \dots, \degdown{(\alpha^{1}_{i_{2} i_{1}})_{i_{2} i_{1}}}  ) \big) \\[1em]
        &{=}
        \sum_{i=0}^{n-1}b_2 (
            \big( \sum_{(i_{n-i+2,\dots,i_n})} f_i( \degdown{\alpha^{n}_{i_{n+1} i_{n}}}, \dots ) \big)_{i_{n+1},i_{n-i+1}},
            \big( \sum_{(i_{2,\dots,i_{n-i}})}f_{n-i}( \dots, \degdown{\alpha^{1}_{i_{2} i_{1}}}  ) \big)_{i_{n-i+1,i_1}}
        )\\[1em]
        &{=}
        \sum_{i=0}^{n-1}
            \big( 
                \sum_{(i_2,\dots,i_n)} 
                b_2(
                f_i( \degdown{\alpha^{n}_{i_{n+1} i_{n}}}, \dots ),
                f_{n-i}( \dots, \degdown{\alpha^{1}_{i_{2} i_{1}}}  )
                )
            \big)_{i_{n+1},i_{1}} \\[1em]
        &{=}
        \sum_{(i_2,\dots,i_n)} 
            \big(
                \sum_{i=0}^{n-1}
                b_2(
                f_i( \degdown{\alpha^{n}_{i_{n+1} i_{n}}}, \dots ),
                f_{n-i}( \dots, \degdown{\alpha^{1}_{i_{2} i_{1}}}  )
                )
            \big)_{i_{n+1},i_{1}} \\[1em]
        &{=}
        \sum_{(i_2,\dots,i_n)} 
            \big(
                \sum_{i=0}^{n-1}
                b_2 \circ (f_i \otimes f_{n-i})( \degdown{\alpha^{n}_{i_{n+1} i_{n}}}, \dots, \degdown{\alpha^{1}_{i_{2} i_{1}}}  )
            \big)_{i_{n+1},i_{1}}
    \end{align*}
    }
    \\ \\
    (D)-term:
    {
    \begin{align*}
        &\phantom{=}b_1 \circ F_n\big(  \degdown{(\alpha^{n}_{i_{n+1} i_{n}})_{i_{n+1} i_{n}}} , \dots , \degdown{(\alpha^{1}_{i_{2} i_{1}})_{i_{2} i_{1}}} \big)\\[1em]
        &{=}
        b_1 \big(
        \sum_{(i_2, \dots, i_n)}
        f_n\big( \degdown{\alpha^{n}_{i_{n+1} i_{n}}}, \dots, \degdown{\alpha^{1}_{i_2 i_{1}}} \big)
    \big)_{i_{n+1}i_1}\\[1em]
    &{=}
        \sum_{(i_2, \dots, i_n)}
        \big(
        b_1 \circ
        f_n\big( \degdown{\alpha^{n}_{i_{n+1} i_{n}}}, \dots, \degdown{\alpha^{1}_{i_2 i_{1}}} \big)
    \big)_{i_{n+1}i_1}\\[1em]
    \end{align*}
    }
    \endgroup
    Thus, we see that Equation \eqref{equation:main} holds for the $F_n$,
    since it is obtained by
    evaluating Equation \eqref{equation:main} for the $f_n$ at $( \degdown{\alpha^{n}_{i_{n+1} i_{n}}}, \dots, \degdown{\alpha^{1}_{i_2 i_{1}}} )$,
    forming a matrix indexed by ${i_{n+1}i_1}$,
    and then summing these evaluations over all indices $(i_2, \dots, i_n)$.
\end{proof}

\subsubsection{Lift to completion by translations}\label{subsubsection:lift_translations}

\begin{notation}
    If $\BC$ is a dg category with translations,
    and $\trans{i}{\beta}{j}$ a morphism in $\BC$,
    we simplify notation w.r.t.\ degree shifts by $\trans{i}{\degdown{\beta}}{j} \coloneqq \degdown{ \trans{i}{{\beta}}{j} }$.
\end{notation}

\begin{construction}\label{construction:lift_to_translations}
Let $\AC$ be a dg category and $\BC$ be a dg category with translations.
    Let $f: \AC_{\infty} \longrightarrow \BC_{\infty}$ be an $A_{\infty}$-functor.
We are going to construct an $A_{\infty}$-functor
\[
    F: (\AC^{[\bullet]})_{\infty} \longrightarrow \BC_{\infty}
\]
extending $f$ w.r.t.\ the natural inclusion $\AC_{\infty} \rightarrow (\AC^{[\bullet]})_{\infty}$.
We set 
\[
    F( A[i] ) := f(A)[i]
\]
on objects.
Moreover, for $n \geq 1$, suppose given a chain of morphisms 
\begin{center}
    \begin{tikzpicture}[label/.style={postaction={
    decorate,
    decoration={markings, mark=at position .5 with \node #1;}},
    mylabel/.style={thick, draw=none, align=center, minimum width=0.5cm, minimum height=0.5cm,fill=white}}]
    \coordinate (r) at (4,0);
    \node (A) {$A_{{n+1}}[i_{n+1}]$};
    \node (B) at ($(A)+(r)$) {$A_{{n}}[i_n]$};
    \node (C) at ($(B) + 0.5*(r)$){$\dots$};
    \node (D) at ($(C)+0.4*(r)$) {$A_{{2}}{[i_2]}$};
    \node (E) at ($(D)+(r)$) {$A_{{1}}[i_1]$};
    
    \draw[->,thick,label={[above]{$\trans{i_{n+1}}{\alpha^n}{i_{n}}$}}] (B) -- (A);
    \draw[->,thick] (C) -- (B);
    
    \draw[->,thick] (D) -- (C);
    \draw[->,thick,label={[above]{$\trans{i_{2}}{\alpha^1}{i_{1}}$}}] (E) -- (D);
    \end{tikzpicture}
\end{center}
in $\AC^{[\bullet]}$. For all $n \geq 1$, we set 
\begin{align*}
    F_n\left(  \degdown{ \trans{i_{n+1}}{\alpha^n}{i_{n}}  }, \dots, \degdown{ \trans{i_{2}}{\alpha^1}{i_{1}} } \right) :=
    \trans{i_{n+1}}{f_n( \degdown{\alpha^n}, \dots, \degdown{\alpha^1} )}{i_1} \cdot (-1)^{\sum_{j=1}^{n+1}i_j}
\end{align*}
on such chains of morphisms.
\end{construction}
\begin{proof}[Correctness of construction]
    Clearly, $F$ extends $f$.
    Furthermore, it is compatible with strict identities.
    For morphisms $\alpha, \beta$ in $\AC$, $i,j,l \in \Z$,
    we compute $b_1$ and $b_2$ within $(\AC^{[\bullet]})_{\infty}$:
    \begin{align*}
        b_1( \degdown{\trans{j}{\alpha}{i}} )
        &= \degdown{-d( \trans{j}{\alpha}{i} )} \\[1em]
        &= (-1)^j\degdown{\trans{j}{-d\alpha}{i}}
        = (-1)^j\trans{j}{b_1\degdown{\alpha}}{i}
    \end{align*}
    and
    \begin{align*}
        b_2( \degdown{\trans{l}{\beta}{j}}, \degdown{\trans{j}{\alpha}{i}} )
        &= (-1)^{\abs{\beta} + j - l}( \degdown{\trans{l}{\beta}{j} \circ \trans{j}{\alpha}{i}} ) \\[1em]
        &= (-1)^{\abs{\beta} + j - l}( \degdown{\trans{l}{\beta \circ \alpha}{i}} )
        = (-1)^{j - l}{\trans{l}{b_2(\degdown{\beta}, \degdown{\alpha})}{i}}.
    \end{align*}
    Now, we will show that Equation \eqref{equation:main} for the $F_n$ can be obtained
    from Equation \eqref{equation:main} for the $f_n$
    by applying
    $\trans{i_{n+1}}{(-)}{i_1} \cdot (-1)^{-i_{n+1}  + \sum_{j=1}^{n+1}i_j}$.
    We will prove this fact for each summand of each term in Equation \eqref{equation:main}:
    \begingroup
    \allowdisplaybreaks
    \\ \\
    $i$-th summand of (A)-term:
    \begin{align*}
        &\phantom{=} F_n \circ ( 1^{\otimes i} \otimes b_1 \otimes 1^{\otimes n - (i+1)}) (\degdown{ \trans{i_{n+1}}{\alpha^n}{i_{n}}  }, \dots, \degdown{ \trans{i_{2}}{\alpha^1}{i_{1}} })\\[1em]
        &{=} F_n (\degdown{ \trans{i_{n+1}}{\alpha^n}{i_{n}}  },
        \dots,
        b_1( \degdown{\trans{i_{n-i+1}}{\alpha^{n-i}}{i_{n-i}}} ),
        \dots,
        \degdown{ \trans{i_{2}}{\alpha^1}{i_{1}} })
        \cdot (-1)^{\sigma(i) + i_{n-i+1} - i_{n+1} }
        \\[1em]
        &{=} F_n (\degdown{ \trans{i_{n+1}}{\alpha^n}{i_{n}}  },
        \dots,
        \trans{i_{n-i+1}}{b_1\degdown{{\alpha^{n-i}}}}{i_{n-i}},
        \dots,
        \degdown{ \trans{i_{2}}{\alpha^1}{i_{1}} })
        \cdot (-1)^{\sigma(i) - i_{n+1} }
        \\[1em]
        &{=} \trans{i_{n+1}}
        {f_n (\degdown{ {\alpha^n} },
        \dots,
        {b_1\degdown{{\alpha^{n-i}}}},
        \dots,
        \degdown{ {\alpha^1} })}
        {i_1}
        \cdot (-1)^{\sigma(i) - i_{n+1} + \sum_{j=1}^{n+1} i_j }
        \\[1em]
        &{=} \trans{i_{n+1}}
        {
            f_n \circ ( 1^{\otimes i} \otimes b_1 \otimes 1^{\otimes n - (i+1)})
            (\degdown{\alpha^n}, \dots, \degdown{\alpha^1})
        }
        {i_1}
        \cdot (-1)^{- i_{n+1} + \sum_{j=1}^{n+1} i_j }
    \end{align*}
    \\ \\
    $i$-th summand of (B)-term:
    {
        \tiny
    \begin{align*}
        &\phantom{=} F_{n-1} \circ ( 1^{\otimes i} \otimes b_2 \otimes 1^{\otimes n - (i+2)}) (\degdown{ \trans{i_{n+1}}{\alpha^n}{i_{n}}  }, \dots, \degdown{ \trans{i_{2}}{\alpha^1}{i_{1}} })\\[1em]
        &{=} F_{n-1} (\degdown{ \trans{i_{n+1}}{\alpha^n}{i_{n}}  },
        \dots,
        b_2( \degdown{\trans{i_{n-i+1}}{\alpha^{n-i}}{i_{n-i}}}, \degdown{\trans{i_{n-i}}{\alpha^{n-i-1}}{i_{n-i-1}}} ),
        \dots,
        \degdown{ \trans{i_{2}}{\alpha^1}{i_{1}} })
        \cdot (-1)^{\sigma(i) + i_{n-i+1} - i_{n+1} }
        \\[1em]
        &{=} F_{n-1} (\degdown{ \trans{i_{n+1}}{\alpha^n}{i_{n}}  },
        \dots,
        \trans{i_{n-i+1}}{b_2(\degdown{\alpha^{n-i}}, \degdown{\alpha^{n-i-1}})}{i_{n-i-1}},
        \dots,
        \degdown{ \trans{i_{2}}{\alpha^1}{i_{1}} })
        \cdot (-1)^{\sigma(i) + i_{n-i} - i_{n+1} }
        \\[1em]
        &{=}
        \trans{i_{n+1}}
        {
        f_{n-1} (\degdown{ {\alpha^n}  },
        \dots,
        {b_2(\degdown{\alpha^{n-i}}, \degdown{\alpha^{n-i-1}})},
        \dots,
        \degdown{ {\alpha^1} })
        }
        {i_{1}}
        \cdot (-1)^{\sigma(i) + i_{n-i} - i_{n+1} + \sum_{j=1, j\neq n-i}^{n+1}i_j }
        \\[1em]
        &{=}
        \trans{i_{n+1}}
        {
            f_{n-1} \circ ( 1^{\otimes i} \otimes b_2 \otimes 1^{\otimes n - (i+2)})
            (\degdown{\alpha^n}, \dots, \degdown{\alpha^1})
        }
        {i_{1}}
        \cdot (-1)^{- i_{n+1} + \sum_{j=1}^{n+1}i_j }
    \end{align*}
    }
    \\ \\
    $i$-th summand of (C)-term:
    \begin{align*}
        &\phantom{=} b_2 \circ (F_i \otimes F_{n-i}) (\degdown{ \trans{i_{n+1}}{\alpha^n}{i_{n}}  }, \dots, \degdown{ \trans{i_{2}}{\alpha^1}{i_{1}} })\\[1em]
        &{=} b_2\big( F_i( \degdown{ \trans{i_{n+1}}{\alpha^n}{i_{n}}  }, \dots ), F_{n-i}( \dots, \degdown{ \trans{i_{2}}{\alpha^1}{i_{1}} } ) \big)\\[1em]
        &{=} b_2\big( \trans{i_{n+1}}{f_i( \degdown{ {\alpha^n} }, \dots )}{i_{n-i+1}}, \trans{i_{n-i+1}}{f_{n-i}( \dots, \degdown{ {\alpha^1} } )}{i_1} \big)
        \cdot (-1)^{\sum_{j=1, j \neq n-i+1}^{n+1}i_j}
        \\[1em]
        &{=} \trans{i_{n+1}}{ b_2\big(f_i( \degdown{ {\alpha^n} }, \dots ), f_{n-i}( \dots, \degdown{ {\alpha^1} } )\big)}{i_1}
        \cdot (-1)^{-i_{n+1} + i_{n-i+1} + \sum_{j=1, j \neq n-i+1}^{n+1}i_j}
        \\[1em]
        &{=} \trans{i_{n+1}} {b_2 \circ (f_i \otimes f_{n-i}) (\degdown{ {\alpha^n} }, \dots, \degdown{ {\alpha^1} }) }{i_{1}}
        \cdot (-1)^{-i_{n+1} + \sum_{j=1}^{n+1}i_j}
    \end{align*}
    \\ \\
    (D)-term:
    \begin{align*}
        &\phantom{=} b_1 \circ F_n (\degdown{ \trans{i_{n+1}}{\alpha^n}{i_{n}}  }, \dots, \degdown{ \trans{i_{2}}{\alpha^1}{i_{1}} })\\[1em]
        &{=} b_1 \big(  \trans{i_{n+1}}{ f_n (\degdown{ {\alpha^n}  }, \dots, \degdown{ {\alpha^1} })  }{i_{1}} \big)
        \cdot (-1)^{\sum_{j=1}^{n+1}i_j}
        \\[1em]
        &{=} \trans{i_{n+1}}{ b_1 \circ f_n (\degdown{ {\alpha^n}  }, \dots, \degdown{ {\alpha^1} })  }{i_{1}}
        \cdot (-1)^{-i_{n+1} + \sum_{j=1}^{n+1}i_j}
    \end{align*}
    \endgroup
\end{proof}

\subsubsection{Lift to pretriangulated hulls}\label{subsubsection:lift_pretr}

\begin{notation}
    If $\BC$ is a dg category with twisted complexes,
    and $\twistc{q}{{\beta}}{p}$ a morphism in $\BC$,
    we simplify notation w.r.t.\ degree shifts by $\twistc{q}{\degdown{\beta}}{p} \coloneqq \degdown{ \twistc{q}{{\beta}}{p} }$.
\end{notation}

\begin{lemma}\label{lemma:maurer_cartan_under_f}
    Let $f: \AC_{\infty} \rightarrow \BC_{\infty}$ be an $A_{\infty}$-functor between dg categories $\AC$, $\BC$.
    Suppose given $A \in \AC$, $A \xrightarrow{p} A$ of degree $1$ satisfying the Maurer Cartan equation.
    For every $n \geq 1$, we set 
    \[
        q_n := \degup{f_n( \degdown{p}, \dots, \degdown{p} )}.
    \]
    If we have $q_n = 0$ for $n \gg 0$, then
    \[
        q := (fA \xrightarrow{\sum_{i = 1}^{\infty}q_i} fA)
    \]
    is of degree $1$ and satisfies the Maurer Cartan equation.
\end{lemma}
\begin{proof}
    The Maurer Cartan equation $dq + q^2 = 0$ in a dg category $\AC$
    is equivalent to $b_1\degdown{q} + b_2(\degdown{q},\degdown{q}) = 0$ in the corresponding $A_{\infty}$-category $\AC_{\infty}$,
    since $\abs{q} = 1$.
    Thus, we compute (without any signs since $|{\degdown{q}}| = |{\degdown{p}}| = 0$):
    \begingroup
    \begin{align*}
        &\phantom{=}b_1\degdown{q} + b_2(\degdown{q},\degdown{q}) \\[1em]
        &{=}\big(\sum_{n=1}^{\infty} b_1\degdown{q_n} \big) + b_2(\sum_{n=1}^{\infty}\degdown{q_n},\sum_{n=1}^{\infty}\degdown{q_n}) \\[1em]
        &{=}\big(\sum_{n=1}^{\infty} b_1 \circ f_n(\underbrace{\degdown{p}, \dots, \degdown{p}}_{\times n}) \big) + 
        \big( \sum_{n=1}^{\infty} \sum_{i=1}^{n-1} b_2(f_i(\underbrace{\degdown{p}, \dots, \degdown{p}}_{\times i}), f_{n-i}(\underbrace{\degdown{p}, \dots, \degdown{p}}_{\times n-i})) \big)\\[1em]
        &{=}
        \sum_{n=1}^{\infty} 
        \big(
        b_1 \circ f_n
        +
        \sum_{i=1}^{n-1} b_2 \circ (f_i \otimes f_{n-i})
        \big)
        (\underbrace{\degdown{p}, \dots, \degdown{p}}_{\times n})\\[1em]
        &{=}
        \sum_{n=1}^{\infty} 
        \big(
            {\sum_{{i = 0}}^{n-1}f_{n} \circ \big( 1^{\otimes i} \otimes b_1 \otimes 1^{\otimes {n-(i+1)}} \big)}
            +
            {\sum_{{i = 0}}^{n-2}f_{n-1} \circ \big( 1^{\otimes i} \otimes b_2 \otimes 1^{\otimes {n-(i+2)}} \big)}
        \big)
        ({\degdown{p}, \dots, \degdown{p}})\\[1em]
        &{=}
        \sum_{n=1}^{\infty} 
        \big(
            {\sum_{{i = 0}}^{n-1} f_n( \underbrace{\degdown{p}, \dots, \degdown{p}}_{\times i}, \underbrace{b_1\degdown{p} + b_2( \degdown{p}, \degdown{p}  )}_{= 0}, \underbrace{\degdown{p}, \dots, \degdown{p}}_{\times n - (i+1)} ) }
        \big) = 0 \qedhere
    \end{align*}
    \endgroup
\end{proof}

\begin{construction}\label{construction:a_inf_lift}
Let $\AC$ be a dg category and $\BC$ be a dg category with twisted complexes.
    Let $f: \AC_{\infty} \rightarrow \BC_{\infty}$ be an $A_{\infty}$-functor.
We are going to construct an $A_{\infty}$-functor
\[
    F: (\Twist_f{(\AC)})_{\infty} \longrightarrow \BC_{\infty}
\]
extending $f$ to the full dg subcategory $\Twist_f{(\AC)} \subseteq \Twist( \AC)$
generated by those twisted complexes $\twistcobj{A}{p}$ for which there exists an $a \geq 0$ (depending on $p$) such that for every $n \geq a$,
$f_n$ vanishes on every input in which $p$ occurs at least $a$ times consecutively.
The extension is understood w.r.t.\ the natural inclusion $\AC_{\infty} \rightarrow (\Twist_f{(\AC)})_{\infty}: A \mapsto \twistcobj{A}{0}$.
We set 
\[
    F( \twistcobj{A}{p} ) := \twistcobj{A}{ \sum_{n=1}^{\infty} \degup{f_n( \degdown{p}^{\otimes n} ) }}
\]
on objects.
Moreover, for $n \geq 1$, suppose given a chain of morphisms 
\begin{center}
    \begin{tikzpicture}[label/.style={postaction={
    decorate,
    decoration={markings, mark=at position .5 with \node #1;}},
    mylabel/.style={thick, draw=none, align=center, minimum width=0.5cm, minimum height=0.5cm,fill=white}}]
    \coordinate (r) at (4,0);
    \node (A) {$\twistcobj{A_{{n+1}}}{p_{n+1}}$};
    \node (B) at ($(A)+(r)$) {$\twistcobj{A_{{n}}}{p_{n}}$};
    \node (C) at ($(B) + 0.5*(r)$){$\dots$};
    \node (D) at ($(C)+0.4*(r)$) {$\twistcobj{A_{{2}}}{p_{2}}$};
    \node (E) at ($(D)+(r)$) {$\twistcobj{A_{{1}}}{p_{1}}$};
    
    \draw[->,thick,label={[above]{$\twistc{p_{n+1}}{\alpha^n}{p_{n}}$}}] (B) -- (A);
    \draw[->,thick] (C) -- (B);
    
    \draw[->,thick] (D) -- (C);
    \draw[->,thick,label={[above]{$\twistc{p_{2}}{\alpha^1}{p_{1}}$}}] (E) -- (D);
    \end{tikzpicture}
\end{center}
in $\Twist( \AC )$. 
For all $i$, we set 
\[q_i := \sum_{n=1}^{\infty} \degup{f_n( \degdown{p_i}^{\otimes n} )}.\]
Moreover, for further simplification, we set
\[
    g^{\mathbf{j}}_n( \degdown{\alpha^n}, \dots, \degdown{\alpha^1} )
    \coloneqq
    f_{(n + \sum_{l = 1}^{n+1} j_l)}\big( \degdown{p_{n+1}}^{\otimes j_{n+1}} \otimes \degdown{\alpha^n} \otimes \degdown{p_{n}}^{\otimes j_{n}} \otimes \dots \otimes \degdown{p_{2}}^{\otimes j_{2}} \otimes \degdown{\alpha^1} \otimes \degdown{p_{1}}^{\otimes j_1} \big)
\]
for all tuples $\mathbf{j} = (j_1, \dots, j_{n+1}) \in \N_0^{n+1}$.
Now, we may set for all $n \geq 1$:
\begin{align*}
    F_n\left(  \degdown{ \twistc{p_{n+1}}{\alpha^n}{p_{n}}  }, \dots, \degdown{ \twistc{p_{2}}{\alpha^1}{p_{1}} } \right)
    \coloneqq
    \twistc{q_{n+1}}
    {
        \left(\sum_{\mathbf{j} \in \N_0^{n+1}} g^{\mathbf{j}}_n( \degdown{\alpha^n}, \dots, \degdown{\alpha^1} ) \right)
    }
    {q_1}
\end{align*}
\end{construction}
\begin{proof}[Correctness of construction]
Due to Lemma \ref{lemma:maurer_cartan_under_f},
$F$ is well-defined on objects.
Due to our restriction to the full dg subcategory $\Twist_f{(\AC)}$,
all occurring infinite sums are actually finite.
Strict identities are clearly respected.
Let $\alpha: A \rightarrow B, \beta: B \rightarrow C$ be morphisms in $\AC$, $p,q,r$ endomorphisms of $A,B,C$, respectively,
of degree $1$ satisfying the Maurer Cartan equation.
We compute $b_1$ and $b_2$ within $\Twist(\AC)_{\infty}$:
\begin{align*}
    b_1( \degdown{\twistc{q}{\alpha}{p}} )
    &= \degdown{-d( \twistc{q}{\alpha}{p} )} \\[1em]
    &= -\twistc{q}{(q \circ \alpha + d\alpha + (-1)^{\abs{\alpha}+1}\alpha \circ p)}{p} \\[1em]
    &= \twistc{q}{b_2(\degdown{q}, \degdown{\alpha})}{p} + \twistc{q}{b_1{\degdown{\alpha}}}{p} + \twistc{q}{b_2(\degdown{\alpha}, \degdown{p})}{p}
\end{align*}
and
\begin{align*}
    b_2( \degdown{\twistc{r}{\beta}{q}}, \degdown{\twistc{q}{\alpha}{p}} )
    &= (-1)^{\abs{\beta}}( \degdown{\twistc{r}{\beta}{q} \circ \twistc{q}{\alpha}{p}} ) \\[1em]
    &= (-1)^{\abs{\beta}}( \degdown{\twistc{r}{\beta \circ \alpha}{p}} )
    = {\twistc{r}{b_2(\degdown{\beta}, \degdown{\alpha})}{p}}.
\end{align*}

Next, we define two sets of operators:
\[
    \Omega^{L}_{\geq n} := 
    \big\{ f_{i + 1 + l} \circ ( 1^{\otimes i} \otimes b_j \otimes 1^{\otimes l}) \mid
    i + j + l \geq n, (j =1,2), (i,l \geq 0)
    \big\},
\]
\[
    \Omega^{R}_{\geq n} := 
    \big\{
    b_s \circ (f_{i_1} \otimes \dots \otimes f_{i_s}) \mid
    i_1 + \dots + i_s \geq n, s \geq 1, (i_1, \dots, i_s \geq 1)
    \big\}.
\]
Furthermore, we define the set of tuples
\[
T := \big\{
    ( \degdown{p_{n+1}}^{\otimes j_{n+1}} \otimes \degdown{\alpha^n} \otimes \degdown{p_{n}}^{\otimes j_{n}} \otimes \dots \otimes \degdown{p_{2}}^{\otimes j_{2}} \otimes \degdown{\alpha^1} \otimes \degdown{p_{1}}^{\otimes j_1} )
    \mid
    j_1, \dots, j_{n+1} \geq 0
\big\}.
\]
By summing over Equation \eqref{equation:main} evaluated at each $t \in T$, we have
\[
    \sum (\Omega^{L}_{\geq n} \circ T) = \sum (\Omega^{R}_{\geq n} \circ T)
\]
where we define the application of an operator on a tuple of non-compatible size to be zero.
We will show that this equation is actually equivalent to the evaluation of Equation \eqref{equation:main} at 
\[\tau := (  \degdown{ \twistc{p_{n+1}}{\alpha^n}{p_{n}}  }, \dots, \degdown{ \twistc{p_{2}}{\alpha^1}{p_{1}} }).\]
\begingroup
\allowdisplaybreaks
(A)-term:
    {
        \tiny
    \begin{align*}
        &\phantom{=} \sum_{i=0}^{n-1} F_n \circ ( 1^{\otimes i} \otimes b_1 \otimes 1^{\otimes n - (i+1)}) (\tau) \\[1em]
        &{=} \sum_{i=0}^{n-1} F_n (\degdown{ \twistc{p_{n+1}}{\alpha^n}{p_{n}}  },
        \dots,
        b_1( \degdown{\twistc{p_{n-i+1}}{\alpha^{n-i}}{p_{n-i}}} ),
        \dots,
        \degdown{ \twistc{p_{2}}{\alpha^1}{p_{1}} })
        \cdot (-1)^{\sigma(i) }
        \\[1em]
        &{=} \sum_{i=0}^{n-1} F_n (\degdown{ \twistc{p_{n+1}}{\alpha^n}{p_{n}}  },
        \dots,
        \twistc{p_{n-i+1}}{ b_2(\degdown{p_{n-i+1}}, \degdown{\alpha^{n-i}}) + b_1( \degdown{\alpha^{n-i}} ) + b_2( \degdown{\alpha^{n-i}}, \degdown{p_{n-i}} )  } {p_{n-i}},
        \dots,
        \degdown{ \twistc{p_{2}}{\alpha^1}{p_{1}} })
        \cdot (-1)^{\sigma(i) }
        \\[1em]
        &{=} \sum_{i=0}^{n-1}
        \twistc{q_{n+1}}
        {
        \left(
        \sum_{\mathbf{j} \in \N_0^{n+1}}
        g_n^{\mathbf{j}} 
        (\degdown{ {\alpha^n}  },
        \dots,
        { b_2(\degdown{p_{n-i+1}}, \degdown{\alpha^{n-i}}) + b_1( \degdown{\alpha^{n-i}} ) + b_2( \degdown{\alpha^{n-i}}, \degdown{p_{n-i}} )  },
        \dots,
        \degdown{{\alpha^1} })
        \cdot (-1)^{\sigma(i) }
        \right)
        }
        {q_1}
        \\[1em]
        &{=} 
        \twistc{q_{n+1}}
        {
        \left(
        \sum \{ \epsilon \in \Omega^{L}_{\geq n} \circ T \mid 
        \text{\small $\epsilon$'s defining term contains $b_1(\degdown{\alpha^i})$, $b_2(\degdown{p_{i+1}},\degdown{\alpha^i})$, or $b_2(\degdown{\alpha^i}, \degdown{p_{i}})$ for some $i$ }
        \}
        \right)}
        {q_1}
    \end{align*}
    }
    \\ \\
    (B)-term:
    {
        \small
    \begin{align*}
        &\phantom{=} \sum_{i=0}^{n-2} F_n \circ ( 1^{\otimes i} \otimes b_2 \otimes 1^{\otimes n - (i+2)}) (\tau) \\[1em]
        &{=} \sum_{i=0}^{n-2} F_n (\degdown{ \twistc{p_{n+1}}{\alpha^n}{p_{n}}  },
        \dots,
        b_2( \degdown{\twistc{p_{n-i+1}}{\alpha^{n-i}}{p_{n-i}}}, \degdown{\twistc{p_{n-i}}{\alpha^{n-i-1}}{p_{n-i-1}}} ),
        \dots,
        \degdown{ \twistc{p_{2}}{\alpha^1}{p_{1}} })
        \cdot (-1)^{\sigma(i) }
        \\[1em]
        &{=} \sum_{i=0}^{n-2} F_n (\degdown{ \twistc{p_{n+1}}{\alpha^n}{p_{n}}  },
        \dots,
        \twistc{p_{n-i+1}}{ b_2(\degdown{ \alpha^{n-i}}, \degdown{\alpha^{n-i-1}}) } {p_{n-i-1}},
        \dots,
        \degdown{ \twistc{p_{2}}{\alpha^1}{p_{1}} })
        \cdot (-1)^{\sigma(i) }
        \\[1em]
        &{=} \sum_{i=0}^{n-2}
        \twistc{q_{n+1}}
        {
        \left(
        \sum_{\mathbf{j} \in \N_0^{n+1}}
        g_n^{\mathbf{j}} 
        (\degdown{ {\alpha^n}  },
        \dots,
        { b_2(\degdown{ \alpha^{n-i}}, \degdown{\alpha^{n-i-1}}) },
        \dots,
        \degdown{{\alpha^1} })
        \cdot (-1)^{\sigma(i) }
        \right)
        }
        {q_1}
        \\[1em]
        &{=} 
        \twistc{q_{n+1}}
        {
        \left(
        \sum \{ \epsilon \in \Omega^{L}_{\geq n} \circ T \mid 
        \text{\small $\epsilon$'s defining term contains $b_2(\degdown{\alpha^i},\degdown{\alpha^{i-1}})$ for some $i$ }
        \}
        \right)}
        {q_1}
    \end{align*}
    }
    Thus, if we subtract the sum of the (A)-term and (B)-term from $\sum (\Omega^{L}_{\geq n} \circ T)$,
    we are left with
    \[
        \sum \{ \epsilon \in \Omega^{L}_{\geq n} \circ T \mid 
        \text{\small $\epsilon$'s defining term contains $b_1(\degdown{p_i})$ or $b_2(\degdown{p_i},\degdown{p_i})$ for some $i$ }
        \}
    \]
    But since $b_1(\degdown{p_i}) + b_2(\degdown{p_i},\degdown{p_i}) = 0$, it easily follows that this sum actually vanishes.
    Thus,
    \[
        \text{(A)-term} + \text{(B)-term} = \twistc{q_{n+1}}{(\sum \Omega^{L}_{\geq n} \circ T )}{q_1}.
    \]
    Next, will take a look at the (C)- and (D)-term.
    \\ \\
    (C)-term:
    \begin{align*}
        &\phantom{=}\sum_{i=0}^{n-1}b_2 \circ (F_i \otimes F_{n-i})(\tau)\\[1em]
        &{=}
        \sum_{i=0}^{n-1}b_2 \big(F_i( \degdown{ \twistc{p_{n+1}}{\alpha^n}{p_{n}}  }, \dots ), F_{n-i}( \dots, \degdown{ \twistc{p_{2}}{\alpha^1}{p_{1}} }  ) \big) \\[1em]
        &{=}
        \sum
        \{
            \epsilon \in \Omega^{R}_{\geq n} \circ T \mid
            \text{\small $\epsilon$ is of the form $b_2(x,y)$ where both terms $x$,$y$ contain an $\degdown{\alpha^i}$ }
        \}
    \end{align*}
    \\ \\
    (D)-term:
    \begin{align*}
        &\phantom{=}b_1 \circ F_n(\tau)\\[1em]
        &{=}
        \sum_{\mathbf{j} \in \N_0^{n+1}}
        b_1\big( \twistc{q_{n+1}}{g_n^\mathbf{j}( \degdown{\alpha^n}, \dots, \degdown{\alpha^1})}{q_1} \big)
        \\[1em]
        &{=}
        \sum_{\mathbf{j} \in \N_0^{n+1}}
        \twistc{q_{n+1}}
        {
        \left(
        { b_2 ( q_{n+1}, g_n^\mathbf{j}( \degdown{\alpha^n}, \dots, \degdown{\alpha^1}) ) }
        +
        { b_1 (g_n^\mathbf{j}( \degdown{\alpha^n}, \dots, \degdown{\alpha^1}) ) }
        +
        { b_2 ( g_n^\mathbf{j}( \degdown{\alpha^n}, \dots, \degdown{\alpha^1}), q_{1} ) }
        \right)
        }
        {q_1}
        \\[1em]
        &{=}
        \sum
        \{
            \epsilon \in \Omega^{R}_{\geq n} \circ T \mid
            \text{\small $\epsilon$ is of the form $b_1(z)$ or $b_2(x,y)$ where only one term $x$,$y$ contains an $\degdown{\alpha^i}$ }
        \}
    \end{align*}
    \endgroup
    Thus,
    \[
        \text{(C)-term} + \text{(D)-term} = \twistc{q_{n+1}}{(\sum \Omega^{R}_{\geq n} \circ T )}{q_1},
    \]
    which is the claim.

\end{proof}

\begin{theorem}\label{theorem:lift_theorem}
    Let $\AC$, $\BC$ be dg categories.
    Let $F: \AC \rightarrow \pretr( \BC )$ be an \Ainf-functor.
    Then the formulas in the Constructions \ref{construction:lift_direct_sums}, \ref{construction:lift_to_translations}, \ref{construction:a_inf_lift}
    give rise to an \Ainf-functor 
    \[F^{\sharp}: \pretr( \AC ) \rightarrow \pretr( \BC )\]
    which corresponds to $F$ under the natural isomorphism
    \[
        \Hom_{\Hqe}( \AC, \pretr(\BC) ) \simeq \Hom_{\Hqe_{\pretr}}( \pretr(\AC), \pretr(\BC) ).
    \]
\end{theorem}
\begin{proof}
    First, we need to check that Construction \ref{construction:a_inf_lift} sends one-sided twisted complexes to one-sided twisted complexes.
    Let
    $\twistcobj{A}{q}  = \twistcobj{(\bigoplus_{i = 1}^l A_i[ t_i ])}{(A_i[ t_i ] \xrightarrow{q_{ji}} A_j[ t_j ])_{ji}}$
    be a one-sided twisted complex
    with $l \in \N_0$, $A_i \in \AC$, $t_i \in \Z$ for $i = 1, \dots, l$.
    Let $H$ denote the induced \Ainf-functor
    $(\AC^{[\bullet]})_{\infty} \rightarrow \pretr( \BC_{\infty} )$,
    and let $G$ denote the induced \Ainf-functor
    $(\AC^{[\bullet], \oplus})_{\infty} \rightarrow \pretr( \BC_{\infty} )$.
    If we apply Construction \ref{construction:a_inf_lift}
    to $\twistcobj{A}{q}$, we get the twisted complex
    $\twistcobj{(\bigoplus_{i = 1}^l F( A_i )[ t_i ])}{\sum_{n = 1}^{\infty} \degup{G_n( \degdown{q}^{\otimes n})}}$.
    By construction of $G$, we have
    \[
        G_n( \degdown{q}^{\otimes n})
        =
        \left(
            \sum_{(i_2, \dots, i_n)} H_n( \degdown{q_{i_{n+1}i_{n}}}, \dots, \degdown{q_{i_2 i_1}} )
        \right)_{i_{n+1}i_1}
    \]
    and if $i_{n+1} \geq i_1$, one of the morphisms $\degdown{q_{i_{n+1}i_{n}}}, \dots, \degdown{q_{i_2 i_1}}$ has to be $0$
    since $q$ is a lower triangular matrix.
    Since $H_n$ is multilinear, it follows that $G_n( \degdown{q}^{\otimes n})$
    is a lower triangular matrix and so is $\sum_{n = 1}^{\infty} \degup{G_n( \degdown{q}^{\otimes n})}$
    as a sum of lower triangular matrices.
    A similar argument can also be used to verify $\pretr( \AC ) \subseteq \Twist_{G}( \AC )$.
    Therefore, we obtain a well-defined \Ainf-functor
    $F: \pretr( \AC ) \rightarrow \pretr( \BC )$.
    
    As described in Section \ref{section:dg_ainf_theory},
    the direction
    \[
        \Hom_{\Hqe}( \AC, \pretr(\BC) ) \leftarrow \Hom_{\Hqe_{\pretr}}( \pretr(\AC), \pretr(\BC) ).
    \]
    is induced by the composition with the embedding
    \[
        \AC \rightarrow \pretr( \AC ): A \mapsto \twistcobj{A}{0}.
    \]
    Now, the claim follows since $F^{\sharp}$ is an extension of $F$ w.r.t.\ this embedding.
\end{proof}

\section{Applications}\label{section:applications}

\subsection{Quivers of maximal path length $2$}\label{subsection:quivers}

In the creation of \Ainf-functors,
it is convenient that degenerated chains of morphisms (in the sense of the nerve of a category)
do not have to be checked within the defining relations of an \Ainf-functor. This fact is made precise in the following lemma.

\begin{lemma}\label{lemma:degenerations}
    Let $\AC$, $\BC$ be dg categories.
    Suppose given a function $f: \Obj_{\AC} \longrightarrow \Obj_{\BC}$
    and a collection of graded maps
    $f_n: S(A_n, A_{n+1}) \otimes \dots \otimes S(A_1, A_2) \longrightarrow S(fA_1, fA_{n+1})$
    for every $n \geq 1$ and objects $A_1, \dots, A_{n+1}$ in $\AC$
    which strictly respect identities in the sense of Definition \ref{definition:a_inf_func}.
    Then these data satisfy Equation \eqref{equation:main} of Remark \ref{remark:main_equation}
    for all chains of morphisms
    \begin{center}
        \begin{tikzpicture}[label/.style={postaction={
        decorate,
        decoration={markings, mark=at position .5 with \node #1;}},
        mylabel/.style={thick, draw=none, align=center, minimum width=0.5cm, minimum height=0.5cm,fill=white}}]
        \coordinate (r) at (2,0);
        \node (A) {$A_{{n+1}}$};
        \node (B) at ($(A)+(r)$) {$A_{{n}}$};
        \node (C) at ($(B) + (r)$){$\dots$};
        \node (D) at ($(C)+(r)$) {$A_{{2}}$};
        \node (E) at ($(D)+(r)$) {$A_{{1}}$};
        
        \draw[->,thick,label={[above]{${\alpha^n}$}}] (B) -- (A);
        \draw[->,thick] (C) -- (B);
        
        \draw[->,thick] (D) -- (C);
        \draw[->,thick,label={[above]{${\alpha^1}$}}] (E) -- (D);
        \end{tikzpicture}
    \end{center}
    for which there exists an $i \in \{1, \dots, n\}$ such that $\alpha^i$ is a multiple of the identity.
\end{lemma}
\begin{proof}
    If $\alpha^i = 0$ then there is nothing to show, thus, we may assume $\alpha^i = \id_{A_i}$ by $k$-linearity.
    For $n = 1$, we have
    \[
        0 = f_1 \circ b_1( \degdown{\id_{A_1}} ) = b_1 \circ f_1( \degdown{\id_{A_1}} ) = b_1( \degdown{\id_{f(A_1)}} ) = 0.
    \]
    For $n \geq 2$, the (A)-term and the (D)-term of Equation \eqref{equation:main} are $0$ since identities are closed morphisms and 
    $f_n$ vanishes whenever one of its entries is an identity.
    In order to show equality between the (B)- and the (C)-terms, we distinguish three cases:
    \\
    Case $i = n$:
    \begin{align*}
        \text{(B)-term:} & & f_{n-1}( b_2( \degdown{\id_{A_n}}, \degdown{\alpha^{n-1}}  ), \degdown{\alpha^{n-2}}, \dots ) = f_{n-1}( \degdown{\alpha^{n-1}}, \dots, \degdown{\alpha^1} ) \\
        \text{(C)-term:} & & b_2( f_1( \degdown{\id_{A_n}} ), \degdown{\alpha^{n-1}}, \dots ) = f_{n-1}( \degdown{\alpha^{n-1}}, \dots, \degdown{\alpha^1} )
    \end{align*}
    \\
    Case $i = 1$:
    \begin{align*}
        \text{(B)-term:} & & (-1)^s f_{n-1}( \degdown{\alpha^{n}}, \dots, b_2( \degdown{\alpha^2}, \degdown{\id_{A_1}} ) ) = (-1)^{s + \abs{\alpha^2}} f_{n-1}( \degdown{\alpha^n}, \dots, \degdown{\alpha^2} ) \\
        \text{(C)-term:} & & b_2( f_{n-1}( \degdown{\alpha^{n}}, \dots, \degdown{\alpha^{2}} ), \degdown{\id_{A_1}} ) = (-1)^t f_{n-1}( \degdown{\alpha^n}, \dots, \degdown{\alpha^2} )
    \end{align*}
    with $s = \sum_{j = 3}^n\abs{\degdown{\alpha^j}}$ and $t = \sum_{j = 2}^n\abs{\degdown{\alpha^j}} + 1$.
    \\ \\
    Case $1 < i < n$:
    the (C)-term is $0$, and the (B)-term is given by the sum of the two terms
    \[
        (-1)^s f_{n-1}( \dots, \degdown{\alpha^{i+1}}, b_2( \degdown{\id_{A_i}}, \degdown{\alpha^{i-1}} ), \dots, \degdown{\alpha^1})
        =
        (-1)^s f_{n-1}( \dots, \degdown{\alpha^{i+1}}, \degdown{\alpha^{i-1}}, \dots, \degdown{\alpha^1})
    \]
    and
    \[
        (-1)^t f_{n-1}( \dots, b_2( \degdown{\alpha^{i+1}}, \degdown{\id_{A_i}} ), \degdown{\alpha^{i-1}}, \dots, \degdown{\alpha^1})
        =
        (-1)^{t + \abs{\alpha^{i+1}}} f_{n-1}( \dots, \degdown{\alpha^{i+1}}, \degdown{\alpha^{i-1}}, \dots, \degdown{\alpha^1})
    \]
    with $s = \sum_{j = 1}^{i-1}\abs{\degdown{\alpha^j}}$ and $t = s + \abs{\alpha^{i+1}} + 1$.
\end{proof}

\begin{theorem}\label{theorem:bounded_by_2}
    Let $\QC$ be the $k$-category defined by a quiver $Q$ with relations 
    such that the length of all its paths is bounded by $2$.
    Let $\BC$ be a dg category.
    Then any functor $f: \QC \rightarrow \Hzero( \BC )$
    can be lifted to an \Ainf-functor $F: \QC \rightarrow \BC$ such that $\Hzero( F ) = f$.
\end{theorem}
\begin{proof}
    We provide an explicit construction of $F$.
    Since the maximal path length in $Q$ is bounded, there are no cycles. In particular, $\End_{\QC}( v )$ is generated by $\id_v$
    for $v \in Q$. We set $F_1( \degdown{\id_v} ) := \degdown{\id_{fv}}$.
    For all pairs of distinct vertices $v, w \in Q$, let
    $(\alpha_i^{v,w})_i$ be a $k$-basis of $\Hom_{\QC}( v, w )$ (we do not specify the index sets in our notation).
    Choose representatives $\beta_i^{v,w} \in Z^0( \Hom_{\BC}( fv, fw ) )$
    of the classes $f( \alpha_i^{v,w} ) \in \Hom_{\Hzero(\BC)}( fv, fw ) )$.
    We set
    \[
        F_1( \degdown{\alpha_i^{v,w}} ) \coloneqq \degdown{\beta_i^{v,w}}.
    \]
    By extending $k$-linearly, this defines $F_1$, and by its definition, passing to homotopy classes yields $f$.
    For any chain of length $2$ formed by the chosen basis vectors (with $v,w,u \in Q$)
    \[
        u \xleftarrow{\alpha_i^{w,u}} w \xleftarrow{\alpha_j^{v,w}} v
    \]
    we have that
    \[
        \degup{F_1( \degdown{\alpha_i^{w,u} \circ {\alpha_j^{v,w}}} )} - 
        \degup{F_1( \degdown{\alpha_i^{w,u}} )} \circ \degup{F_1( \degdown{\alpha_j^{v,w}} )}
    \]
    is a coboundary since $f$ is a functor.
    Thus, we find $h_{i,j}^{v,w,u} \in \Hom_{\BC}^{-1}( fv, fu )$ such that
    \[
        d( h_{i,j}^{v,w,u} ) = 
        \degup{F_1( \degdown{\alpha_i^{w,u} \circ {\alpha_j^{v,w}}} )} - 
        \degup{F_1( \degdown{\alpha_i^{w,u}} )} \circ \degup{F_1( \degdown{\alpha_j^{v,w}} )}
    \]
    an we set
    \[
        F_2( \degdown{\alpha_i^{w,u}}, \degdown{\alpha_j^{v,w}} ) \coloneqq -\degdown{  h_{i,j}^{v,w,u} }.
    \]
    On every chain of morphisms involving an identity, we set $F_2$ to $0$. We extend $k$-linearly and obtain our map $F_2$. Moreover, all $F_n$ for $n \geq 3$ are set to $0$.
    
    We claim that $F_n$ for $n \in \N$ give rise to an \Ainf-functor.
    For this, we need to verify Equation \eqref{equation:main} of Remark \ref{remark:main_equation}.
    The case $n = 1$ is satisfied since $F_1$ sends closed morphisms to closed morphisms.
    Every chain of consecutive morphisms in $\AC$ that contain a multiple of an identity
    does not have to be tested by Lemma \ref{lemma:degenerations}.
    But since the maximal path length of $Q$ is bounded by $2$,
    this leaves us with testing chains of morphisms of length $2$
    formed by the chosen basis vectors. Since in this case,
    the Equation \eqref{equation:main} is given by
    \[
        f_1 \circ b_2 = b_2 \circ (f_1 \otimes f_1) + b_1 \circ f_2,
    \]
    the claim follows directly from our construction.
\end{proof}

From the construction in the proof of Theorem \ref{theorem:bounded_by_2} we see that there are lots of choices involved.
In the next example, we demonstrate that different choices may lead to significant differences in the result.

\begin{example}\label{example:counterex_generators}
    Let $\QC$ be the $k$-category given by the quiver
    \vspace{1em}
    \begin{center}
        \begin{tikzpicture}[label/.style={postaction={
        decorate,
        decoration={markings, mark=at position .5 with \node #1;}},
        mylabel/.style={thick, draw=none, align=center, minimum width=0.5cm, minimum height=0.5cm,fill=white}}]
        \coordinate (r) at (3,0);
        \node (A) {$A$};
        \node (B) at ($(A)+(r)$) {$B$};
        \node (C) at ($(B) + (r)$){$C$};
        \draw[->,thick,label={[above]{$\beta$}},transform canvas={yshift=0em}] (A) -- (B);
        \draw[->,thick,label={[above]{$\gamma$}},transform canvas={yshift=0em}] (B) -- (C);
        \end{tikzpicture}
    \end{center}
    subject to the relation
    \[
     \gamma \circ \beta = 0.
    \]
    Let $\BC$ be the dg category
    given by the dg quiver
    \vspace{1em}
    \begin{center}
        \begin{tikzpicture}[label/.style={postaction={
        decorate,
        decoration={markings, mark=at position .5 with \node #1;}},
        mylabel/.style={thick, draw=none, align=center, minimum width=0.5cm, minimum height=0.5cm,fill=white}}]
        \coordinate (r) at (3,0);
        \node (A) {$A$};
        \node (B) at ($(A)+(r)$) {$B$};
        \node (C) at ($(B) + (r)$){$C$};
        \draw[->,thick,label={[above]{$\beta$}},transform canvas={yshift=-0em}] (A) -- (B);
        \draw[->,thick,label={[above]{$\gamma$}},transform canvas={yshift=0em}] (B) -- (C);
        
        \node (Ad) at ($(A)+(0,-0.15)$) {\phantom{$A$}};
        \node (Cd) at ($(C)+(0,-0.15)$) {\phantom{$C$}};
        \node (Ad2) at ($(A)+(-0.15,-0.45)$) {\phantom{$A$}};
        \node (Cd2) at ($(C)+(0.15,-0.45)$) {\phantom{$C$}};
        \draw[bend left,->,thick,label={[mylabel]{$\omega$}},out=-25,in=-155] (Ad) to (Cd);
        \draw[bend left,->,thick,label={[mylabel]{$\tau$}},out=-25,in=-155] (Ad2) to (Cd2);
        \end{tikzpicture}
    \end{center}
    with $\beta$, $\gamma$ closed morphisms of degree $0$, $\abs{\omega} = \abs{\tau} = -1$,
    and
    \[
        d( \omega ) = d( \tau ) = \gamma \circ \beta.
    \]
    Let
    $
        f: \QC \rightarrow \Hzero( \BC )
    $
    be the functor induced by mapping $\beta$, $\gamma$ to their corresponding classes in $\Hzero( \BC )$.
    It is well-defined since the classes $\resclass{\beta}, \resclass{\gamma} \in \Hzero( \BC )$ satisfy the required relation
    $\resclass{\gamma} \circ \resclass{\beta} = 0$.
    We use the construction in the proof of Theorem \ref{theorem:bounded_by_2}
    in order to define two different lifts $F, G$ of $f$ to the \Ainf-level.
    For both $F_1$ and $G_1$, we set 
    \begin{align*}
        \degdown{\beta} &\mapsto \degdown{\beta} \\
        \degdown{\gamma} &\mapsto \degdown{\gamma}
    \end{align*}
    The difference between $F$, $G$ comes from the following choices:
    \[
        F_2( ( \degdown{\gamma}, \degdown{\beta }) ) \coloneqq -\degdown{\omega}
        \hspace{4em}
        G_2( ( \degdown{\gamma}, \degdown{\beta }) ) \coloneqq -\degdown{\tau}.
    \]
    By Theorem \ref{theorem:bounded_by_2}, we have $\Hzero( F ) = \Hzero( G ) = f$.
    However, $F$ and $G$ behave differently with respect to the passage to pretriangulated hulls.
    Let 
    \[
        F^{\sharp}: \pretr( \QC ) \rightarrow \pretr( \BC )
        \hspace{3em}
        \text{and}
        \hspace{3em}
        G^{\sharp}: \pretr( \QC ) \rightarrow \pretr( \BC )
    \]
    be the \Ainf-functors induced by the universal property of pretriangulated hulls applied to 
    \[
        \QC \xrightarrow{F} \BC \rightarrow \pretr( \BC )
        \hspace{3em}
        \text{and}
        \hspace{3em}
        \QC \xrightarrow{G} \BC \rightarrow \pretr( \BC ).
    \]
    We claim that
    \[
        \Hzero( F^{\sharp} ) \not\simeq \Hzero( G^{\sharp} )
    \]
    as exact functors between triangulated categories.
    To see this, we define the complex
    \begin{center}
        \begin{tikzpicture}[label/.style={postaction={
        decorate,
        decoration={markings, mark=at position .5 with \node #1;}},
        mylabel/.style={thick, draw=none, align=center, minimum width=0.5cm, minimum height=0.5cm,fill=white}}]
        \coordinate (r) at (4,0);
        \node (A) {$A$};
        \node (D) at ($(A)-0.2*(r)$) {$\mathcal{C} \coloneqq$};
        \node (B) at ($(A)+(r)$) {$B[-1]$};
        \node (C) at ($(B) + (r)$){$C[-2]$};
        \draw[->,thick,label={[above]{${\trans{-1}{\beta}{0}}$}}] (A) -- (B);
        \draw[->,thick,label={[above]{${\trans{-2}{\gamma}{-1}}$}}] (B) -- (C);
        \end{tikzpicture}
    \end{center}
    in $\pretr( \QC ) \simeq \Chbdg( \QC )$.
    The explicit constructions in Theorem \ref{theorem:lift_theorem} allow us to compute the images of $\mathcal{C}$
    under $F^{\sharp}$ and $G^{\sharp}$.
    They turn out to be the twisted complexes
    \begin{center}
        \begin{tikzpicture}[label/.style={postaction={
        decorate,
        decoration={markings, mark=at position .5 with \node #1;}},
        mylabel/.style={thick, draw=none, align=center, minimum width=0.5cm, minimum height=0.5cm,fill=white}}]
        \coordinate (r) at (4,0);
        \node (A) {$A$};
        \node (D) at ($(A)-0.3*(r)$) {$F^{\sharp}(\mathcal{C}) =$};
        \node (B) at ($(A)+(r)$) {$B[-1]$};
        \node (C) at ($(B) + (r)$){$C[-2]$};
        \draw[->,thick,label={[above]{${\trans{-1}{\beta}{0}}$}}] (A) -- (B);
        \draw[->,thick,label={[above]{${\trans{-2}{\gamma}{-1}}$}}] (B) -- (C);
        \draw[bend left,->,thick,label={[mylabel]{$\trans{-2}{\omega}{0}$}},out=-25,in=-155] (A) to (C);
        \end{tikzpicture}
    \end{center}
    and
    \begin{center}
        \begin{tikzpicture}[label/.style={postaction={
        decorate,
        decoration={markings, mark=at position .5 with \node #1;}},
        mylabel/.style={thick, draw=none, align=center, minimum width=0.5cm, minimum height=0.5cm,fill=white}}]
        \coordinate (r) at (4,0);
        \node (A) {$A$};
        \node (D) at ($(A)-0.3*(r)$) {$G^{\sharp}(\mathcal{C}) =$};
        \node (B) at ($(A)+(r)$) {$B[-1]$};
        \node (C) at ($(B) + (r)$){$C[-2]$};
        \draw[->,thick,label={[above]{${\trans{-1}{\beta}{0}}$}}] (A) -- (B);
        \draw[->,thick,label={[above]{${\trans{-2}{\gamma}{-1}}$}}] (B) -- (C);
        \draw[bend left,->,thick,label={[mylabel]{$\trans{-2}{\tau}{0}$}},out=-25,in=-155] (A) to (C);
        \end{tikzpicture}
    \end{center}
    Note that these diagrams depict the twisted complexes
    \[
        F^{\sharp}(\mathcal{C}) =\twistcobj{(A \oplus B[-1] \oplus C[-2])}{q_{\omega}}
        \hspace{2em}
        \text{and}
        \hspace{2em}
        G^{\sharp}(\mathcal{C}) =\twistcobj{(A \oplus B[-1] \oplus C[-2])}{q_{\tau}}
    \]
    with
    \[
    q_\omega \coloneqq \begin{pmatrix} \cdot & \cdot & \cdot \\ {\trans{-1}{\beta}{0}} & \cdot & \cdot \\ \trans{-2}{\omega}{0} & {\trans{-2}{\gamma}{-1}} & \cdot \end{pmatrix}
        \hspace{2em}
        \text{and}
        \hspace{2em}
    q_\tau \coloneqq \begin{pmatrix} \cdot & \cdot & \cdot \\ {\trans{-1}{\beta}{0}} & \cdot & \cdot \\ \trans{-2}{\tau}{0} & {\trans{-2}{\gamma}{-1}} & \cdot \end{pmatrix}.
    \]
    We claim that 
    \[F^{\sharp}(\mathcal{C}) \not\cong G^{\sharp}(\mathcal{C})\]
    in $\Hzero( \pretr( \BC ) )$.
    Since all degree $-1$ morphisms between $F^{\sharp}(\mathcal{C})$ and $G^{\sharp}(\mathcal{C})$ are zero,
    it suffices to prove the claim in 
    $Z^0( \pretr( \BC ) )$.
    Any degree $0$ morphism $F^{\sharp}(\mathcal{C}) \rightarrow G^{\sharp}(\mathcal{C})$
    is induced by a block diagonal matrix
    \[
        \mu_{a,b,c}: A \oplus B[-1] \oplus C[-2]
        \xrightarrow{(a \oplus b \oplus c)}
        A \oplus B[-1] \oplus C[-2]
    \]
    with
    $a,b,c \in k$.
    Using the formula for the differential in Lemma \ref{lemma:formulas_twisted_complexes},
    we may compute that the underlying matrix of $d( \twistc{q_\tau}{\mu_{a,b,c}}{q_\omega} )$
    is given by
    \begin{align*}
         \begin{pmatrix} \cdot & \cdot & \cdot \\ a({\trans{-1}{\beta}{0}}) & \cdot & \cdot \\ a(\trans{-2}{\omega}{0}) & b({\trans{-2}{\gamma}{-1}}) & \cdot \end{pmatrix}
         -
        \begin{pmatrix} \cdot & \cdot & \cdot \\ b({\trans{-1}{\beta}{0}}) & \cdot & \cdot \\ c(\trans{-2}{\tau}{0}) & c({\trans{-2}{\gamma}{-1}}) & \cdot \end{pmatrix}.
    \end{align*}
    Since $\omega$ and $\tau$ are linearly independent, it follows that $\mu_{a,b,c}$ gives rise to a closed morphism
    if and only if it is $0$.
    But since both twisted complexes are not isomorphic to $0$ in $Z^0( \pretr( \BC ) )$, they cannot be isomorphic.
\end{example}

\begin{corollary}\label{corollary:warning}
    There are triangulated categories $\TC_1$, $\TC_2$, 
    a full subcategory $\AC \subseteq \TC_1$  spanned by a full strong exceptional sequence
    and exact functors
    $F,G: \TC_1 \rightarrow \TC_2$ with $F|_{\AC} \simeq G|_{\AC}$, but $F \not\simeq G$.
\end{corollary}
\begin{proof}
    See Example \ref{example:counterex_generators}.
\end{proof}

Now, we come back to the example $(\OS(-1) \otimesL - ): \Db( \Pro^2 ) \rightarrow \Db( \Pro^2 )$ 
of Subsection \ref{subsection:difficulty}.
Since $\BLineBundles{2}$ is induced by a quiver of maximal path length $2$,
we may apply the construction in Theorem \ref{theorem:bounded_by_2}
and obtain an \Ainf-functor
\[
    F: \BLineBundles{2} \rightarrow \pretr( \BLineBundles{2} )
\]
such that the diagram 
\begin{center}
    \begin{tikzpicture}[label/.style={postaction={
    decorate,
    decoration={markings, mark=at position .5 with \node #1;}},
    mylabel/.style={thick, draw=none, align=center, minimum width=0.5cm, minimum height=0.5cm,fill=white}}]
    \coordinate (r) at (5,0);
    \coordinate (d) at (0,-2);
    \node (L) {$\BLineBundles{2}$};
    \node (A) at ($(L) + (r)$) {$\Kb( \BLineBundles{2} )$};
    
    \node (L2) at ($(L) + (d)$) {$\Db( \Pro^{2} )$};
    \node (A2) at ($(A) + (d)$) {$\Db( \Pro^2)$};
    
    \draw[->,thick] (L) --node[above]{$\Hzero( F )$} (A);
    \draw[->,thick] (L2) --node[above]{$(\OS(-1) \otimesL - )$} (A2);
    
    \draw[->,thick] (L) -- (L2);
    \draw[->,thick] (A) -- (A2);
    
    \end{tikzpicture}
\end{center}
commutes up to natural isomorphism.
However, Corollary \ref{corollary:warning}
gives us a warning that in such a situation, it is not a priori clear that
the \Ainf-functor $F^{\sharp}: \pretr(\BLineBundles{2}) \rightarrow \pretr( \BLineBundles{2} )$ induced by the universal property of pretriangulated hulls applied to $F$
gives us our desired model of $(\OS(-1) \otimesL - )$.
Luckily, in this example, the situation is sufficiently special.

\begin{theorem}\label{theorem:special_situation}\cite[cf.\ Theorem 4.15]{Genovese16}
    Let $\AC$ be a $k$-category, $\BC$ be a pretriangulated dg category.
    Let furthermore $\Hzero( \pretr( \AC ) )$ and $\Hzero( \BC )$ be idempotent complete.
    Let $F, G: \pretr( \AC ) \rightarrow \BC$ be \Ainf-functors such that
    \[
        \CH^j( \Hom_{\BC}( FA, FA' )  ) \cong 0
    \]
    for all $j < 0$, $A, A' \in \AC$.
    Then $\Hzero( F )|_{\AC} \simeq \Hzero( G )|_{\AC}$
    implies $F = G$ in $\Hqe$.
    In particular, $\Hzero( F ) \simeq \Hzero( G )$.
\end{theorem}

Since 
\begin{enumerate}
    \item all $\Ext$-groups between coherent sheaves are non-negative,
    \item $\Db( \Pro^n )$ is idempotent complete (since it is the bounded derived category of an abelian category \cite{BalSch01}),
    \item there exists at least one \Ainf-functor which models tensoring with $\OS(-1)$ (Theorem \ref{theorem:FM_as_Ainf}),
\end{enumerate}
we may apply Theorem \ref{theorem:special_situation}
to justify that our lift $F^{\sharp}: \pretr(\BLineBundles{2}) \rightarrow \pretr( \BLineBundles{2} )$
really models tensoring with $\OS(-1)$.

\subsection{Fourier-Mukai transforms for projective spaces revisited}\label{subsection:fm_for_proj}

We come back to the question posed in Section \ref{section:products}:
given $a,b \in \N_0$, how do we explicitly rewrite any Fourier-Mukai transform
$\Db( \Pro^a ) \rightarrow \Db( \Pro^b )$ as a functor
between the homotopy categories of the Beilinson quiver with relations
$\Kb( \BLineBundles{a} ) \rightarrow \Kb( \BLineBundles{b} )$?
In Subsections \ref{subsection:pullback} and \ref{subsection:direct_image},
we saw that the pullback functor
$\Db( \Pro^a ) \rightarrow \Db( \Pro^a \times \Pro^b )$
and the direct image functor 
$\Db( \Pro^a \times \Pro^b ) \rightarrow \Db( \Pro^b )$
of a projection can be easily rewritten,
thus, we are left with the tensor product
by an arbitrary object in $\Db( \Pro^a \times \Pro^b )$.

Let $\LineBundles \subseteq \Db( \Pro^a \times \Pro^b )$
be the full subcategory generated by all line bundles.
Recall that in Section \ref{section:products},
we defined the full subcategory
\[
    \LineBundlesab = \{ \OS(i) \boxtimes \OS(j) \mid i = -a, \dots, 0, ~~j = -b, \dots 0 \} \subseteq \LineBundles.
\]
In addition, we will need the full subcategory
\[
    \LineBundlesabdouble \coloneqq \{ \OS(i) \boxtimes \OS(j) \mid i = -2a, \dots, 0, ~~j = -2b, \dots 0 \} \subseteq \LineBundles.
\]
Let $K$ be a bounded complex whose objects are direct sums of line bundles in $\LineBundlesab$.
The usual tensor product of complexes gives a functor
$(- \otimes K): \Kb( \LineBundles ) \rightarrow \Kb( \LineBundles )$.
Since tensoring with $K$ restricted to $\Kb( \LineBundlesab)$ yields an output in $\Kb( \LineBundlesabdouble)$,
and since tensoring with bounded complexes of vector bundles respects acyclic complexes \cite[Chapter 3]{Huy_FM},
we end up with a diagram (commutative up to natural isomorphism):
\begin{center}
    \begin{tikzpicture}[label/.style={postaction={
    decorate,
    decoration={markings, mark=at position .5 with \node #1;}},
    mylabel/.style={thick, draw=none, align=center, minimum width=0.5cm, minimum height=0.5cm,fill=white}}]
    \coordinate (r) at (5,0);
    \coordinate (d) at (0,-2);
    
    \node (L2) {$\Kb( {\LineBundles} )$};
    \node (A2) at ($(L2) + (r)$) {$\Kb( {\LineBundles} )$};
    
    \node (L3) at ($(L2) + (d)$) {$\Kb( {\LineBundlesab} )$};
    \node (A3) at ($(A2) + (d)$) {$\Kb( {\LineBundlesabdouble} )$};
    \node (B3) at ($(A3) + (r)$) {$\Kb( {\LineBundlesab} )$};
    
    \node (L4) at ($(L3) + (d)$) {$\Db( \Pro^{a} \times \Pro^{b} )$};
    \node (A4) at ($(A3) + (d)$) {$\Db( \Pro^{a} \times \Pro^{b} )$};

    \draw[->,thick] (L2) --node[above]{$(- \otimes K)$} (A2);
    \draw[->,thick] (L3) --node[above]{$(- \otimes K)$} (A3);
    \draw[->,thick] (L4) --node[above]{$(- \otimesL K)$} (A4);

    \draw[left hook->,thick] (B3) -- (A3);
    \draw[->,thick] (B3) --node[above,rotate=20]{$\sim$} node[below]{$\iota$} (A4);
    
    \draw[left hook->,thick] (L3) -- (L2);
    \draw[right hook->,thick] (A3) -- (A2);
    \draw[->,thick] (L3) --node[above,rotate=90]{$\sim$}  (L4);
    \draw[->,thick] (A3) --node[right]{$\pi$} (A4);
    \end{tikzpicture}
\end{center}

\begin{lemma}
    Let 
    \[P: \LineBundlesabdouble \rightarrow \pretr( \LineBundlesab )\]
    be an \Ainf-functor
    such that 
    \[\Hzero( P ) \simeq (\iota^{-1} \circ \pi)|_{\LineBundlesabdouble}.\]
    Then the diagram
    \begin{center}
        \begin{tikzpicture}[label/.style={postaction={
        decorate,
        decoration={markings, mark=at position .5 with \node #1;}},
        mylabel/.style={thick, draw=none, align=center, minimum width=0.5cm, minimum height=0.5cm,fill=white}}]
        \coordinate (r) at (7,0);
        \coordinate (d) at (0,-2);
        \node (L) {$\Kb( \LineBundlesab )$};
        \node (A) at ($(L) + (r)$) {$\Kb( \LineBundlesab )$};
        
        \node (L2) at ($(L) + (d)$) {$\Db( \Pro^{a} \times \Pro^{b}) $};
        \node (A2) at ($(A) + (d)$) {$\Db( \Pro^{a} \times \Pro^{b}) $};
        
        \draw[->,thick] (L) --node[above]{$\Hzero( P^{\sharp})\circ (- \otimes K)$} (A);
        \draw[->,thick] (L2) --node[above]{$(- \otimesL K)$} (A2);
        
        \draw[->,thick] (L) --node[left]{$\iota$} (L2);
        \draw[->,thick] (A) --node[right]{$\iota$} (A2);
        
        \end{tikzpicture}
    \end{center}
    commutes up to natural isomorphism, where 
    \[P^{\sharp}: \pretr( \LineBundlesabdouble ) \rightarrow \pretr( \LineBundlesab )\]
    is the \Ainf-functor induced by the universal property of pretriangulated hulls applied to $P$.
\end{lemma}
\begin{proof}
    The functor $\pi: \Kb( \LineBundlesabdouble ) \rightarrow \Db( \Pro^{a} \times \Pro^{b})$
    arises as the composition
    \[
        \Kb( \LineBundlesabdouble ) \rightarrow \Kb( \Coh( \Pro^a \times \Pro^b ) ) \rightarrow \Db( \Pro^{a} \times \Pro^{b})
    \]
    of a full embedding and a Verdier localization functor.
    Since such functors can be lifted to the dg level \cite[Section 4.4]{Kel06},
    and since dg enhancements of $\Db( \Pro^{a} \times \Pro^{b})$ are strongly unique in our setup (Subsection \ref{subsection:inducing_fm}),
    it follows that
    there exists an \Ainf-functor
    \[
        \Pi: \pretr( \LineBundlesabdouble ) \rightarrow \pretr( \LineBundlesab )
    \]
    with $\Hzero(\Pi) \simeq \iota^{-1} \circ \pi$.
    The category $\Kb( \LineBundlesabdouble )$ is idempotent complete
    since it can be seen as the bounded derived category of finitely generated modules over an acyclic quiver with relations.
    Thus, we may apply Theorem \ref{theorem:special_situation} and deduce
    $\Hzero(P^{\sharp}) \simeq \Hzero(\Pi)$.
    Thus, the claim follows from the diagram depicted above this lemma.
\end{proof}

We sum up our modeling of Fourier-Mukai transforms for projective spaces.

\begin{theorem}\label{theorem:fm_for_pn}
    Let $K \in \Kb( \BLineBundles{a,b} ) \simeq \Db( \Pro^a \times \Pro^b )$, and let $\phi_K: \Db( \Pro^a ) \rightarrow \Db( \Pro^b )$
    denote the Fourier-Mukai transform with kernel $K$.
    Let $\overline{F}: \Kb( \BLineBundles{a} ) \rightarrow \Kb( \BLineBundles{a,b} )$
    denote the functor defined in Subsection \ref{subsection:pullback},
    and let $\overline{G}: \Kb( \BLineBundles{a,b} ) \rightarrow \Kb( \BLineBundles{b} )$
    denote the functor defined in Subsection \ref{subsection:direct_image}
    (but associated to the second projection $q: \Pro^a \times \Pro^b \rightarrow \Pro^b$ instead of the first one).
    Let 
    \[P: \LineBundlesabdouble \rightarrow \pretr( \LineBundlesab )\]
    be an \Ainf-functor
    such that 
    \[\Hzero( P ) \simeq (\iota^{-1} \circ \pi)|_{\LineBundlesabdouble}.\]
    Then the diagram 
    \begin{center}
        \begin{tikzpicture}[label/.style={postaction={
        decorate,
        decoration={markings, mark=at position .5 with \node #1;}},
        mylabel/.style={thick, draw=none, align=center, minimum width=0.5cm, minimum height=0.5cm,fill=white}}]
        \coordinate (r) at (9,0);
        \coordinate (d) at (0,-2);
        \node (L) {$\Kb( \BLineBundles{a} )$};
        \node (A) at ($(L) + (r)$) {$\Kb( \BLineBundles{b} )$};
        
        \node (L2) at ($(L) + (d)$) {$\Db( \Pro^{a}) $};
        \node (A2) at ($(A) + (d)$) {$\Db( \Pro^{b}) $};
        
        \draw[->,thick] (L) --node[above]{$\overline{G} \circ \Hzero( P^{\sharp})\circ (- \otimes K) \circ \overline{F}$} (A);
        \draw[->,thick] (L2) --node[above]{$\phi_K$} (A2);
        
        \draw[->,thick] (L) --node[rotate = 90,above]{$\sim$} (L2);
        \draw[->,thick] (A) --node[rotate = 90,below]{$\sim$} (A2);
        
        \end{tikzpicture}
    \end{center}
    commutes up to natural isomorphism.
\end{theorem}

\begin{remark}
    It follows that the task of finding an explicit formula for expressing
    an arbitrary Fourier-Mukai transform $\Db( \Pro^a ) \rightarrow \Db( \Pro^b )$ in terms of $\BLineBundles{a}$ and $\BLineBundles{b}$ boils down
    to finding an explicit formula for an \Ainf-functor
    $P: \LineBundlesabdouble \rightarrow \pretr( \LineBundlesab )$
    with $\Hzero( P ) \simeq \iota^{-1} \circ \pi$.
    Note that this task is independent of any prescribed Fourier-Mukai kernel.
\end{remark}

%% file: ainf.bbl
\def\cprime{$'$} \def\cprime{$'$} \def\cprime{$'$} \def\cprime{$'$}
  \def\cprime{$'$}
\providecommand{\bysame}{\leavevmode\hbox to3em{\hrulefill}\thinspace}
\providecommand{\MR}{\relax\ifhmode\unskip\space\fi MR }
\providecommand{\MRhref}[2]{%
  \href{http://www.ams.org/mathscinet-getitem?mr=#1}{#2}
}
\providecommand{\href}[2]{#2}